\documentclass[preprint,11pt]{elsarticle}

\usepackage{amsfonts, amsmath, amscd}
\usepackage[psamsfonts]{amssymb}

\usepackage{amssymb}

\usepackage{pb-diagram}

\usepackage[all,cmtip]{xy}

\usepackage[usenames]{color}

\usepackage{amsthm}

\newtheorem{theorem}{Theorem}[section]
\newtheorem{lemma}[theorem]{Lemma}
\newtheorem{corollary}[theorem]{Corollary}

\newtheorem{proposition}[theorem]{Proposition}
\newtheorem{example}[theorem]{Example}

\theoremstyle{definition}
\newtheorem{definition}[theorem]{Definition}
\newtheorem{problem}[theorem]{Problem}
\theoremstyle{remark}
\newtheorem{remark}[theorem]{Remark}

\numberwithin{equation}{section}
\numberwithin{theorem}{section}

\newcommand{\e}{\varepsilon}
\newcommand{\w}{\omega}

\newcommand{\NN}{\mathbb{N}}

\newcommand{\IR}{\mathbb{R}}

\newcommand{\xxx}{\mathbf{x}}

\newcommand{\KK}{\mathcal{K}}

\newcommand{\AAA}{\mathcal{A}}

\newcommand{\Ra}{\Rightarrow}
\newcommand{\La}{\Leftarrow}
\newcommand{\LRa}{\Leftrightarrow}

\newcommand{\Id}{\mathsf{id}}

\newcommand{\CC}{C_k}

\newcommand{\SM}{{\setminus}}

\def\cM{\mathcal{M}}

\def\cB{\mathcal{B}}
\def\cU{\mathcal{U}}
\def\cV{\mathcal{V}}
\def\om{\omega}
\def\bt{\beta}
\def\kr{$k_\IR$}
\def\sr{$s_\IR$}

\def\:{\colon}

\def\si{\sigma}
\def\Si{\Sigma}

\def\rst#1{{\restriction}_{#1}}

\newenvironment{new-content}
    {\color{blue}}
    {}

\begin{document}

\begin{frontmatter}

\title{On  $k_\IR$-spaces and $s_\IR$-spaces}

\author{Saak Gabriyelyan}
\ead{saak@math.bgu.ac.il}
\address{Department of Mathematics, Ben-Gurion University of the Negev, Beer-Sheva, P.O. 653, Israel}

\author{Evgenii Reznichenko}
\ead{erezn@inbox.ru}
\address{Department of Mathematics, Lomonosov Mosow State University, Moscow, Russia}

\begin{abstract}
We give new characterizations of spaces $X$ which are  $k_\IR$-spaces or $s_\IR$-spaces. Applying the obtained results we provide some sufficient and necessary conditions on $X$ for which $C_p(X)$ is a $k_\IR$-space or an $s_\IR$-space. It is proved that $C_p(X)$ is a $k_\IR$-space for any  space $X$ with one non-isolated point; if, in addition, $|X|$ is not sequential, then $C_p(X)$ is even an $s_\IR$-space. Under $(CH)$, it is shown that there exists a separable metrizable space $X$ such that $C_p(X)$ is an Ascoli space but not a $k_\IR$-space.
\end{abstract}

\begin{keyword}
Ascoli space \sep $k_\IR$-space \sep $s_\IR$-space \sep  $k^{\#}$-completion

\MSC[2010] 54A05 \sep  54B05 \sep   54C35 \sep 54D50

\end{keyword}

\end{frontmatter}


\section{Introduction}


All spaces are assumed to be Tychonoff. We denote by $C_p(X)$ and $\CC(X)$ the space $C(X)$ of all continuous real-valued functions on a space $X$ endowed with the pointwise topology or the compact-open topology, respectively. Recall that a space $X$ is called a {\em $k$-space} if for each non-closed subset $A\subseteq X$ there is a compact subset $K\subseteq X$ such that $A\cap K$ is not closed in $K$. The space $X$ is a {\em $k_\IR$-space} if every $k$-continuous function $f:X\to\IR$ is continuous (recall that $f$ is {\em $k$-continuous} if each restriction of $f$ to any compact set $K\subseteq X$ is continuous). Each $k$-space is a $k_\IR$-space, but the converse is false in general. $k$-spaces and $k_\IR$-spaces are widely studied in General Topology (see, for example, the classical book of Engelking \cite{Eng} or  Michael's article \cite{Mi73}). The meaning of $k_\IR$-spaces is explained by the following characterization (see for example Theorem 5.8.7 of \cite{NaB}).

\begin{theorem} \label{t:kR-classical}
A space $X$ is a $k_\IR$-space if, and only if, the space $\CC(X)$ is complete.
\end{theorem}
Another known characterization of $k_\IR$-spaces is obtained in \cite{Banakh-Survey} which will be re-proved below.
\smallskip

One of the basic theorems in Analysis is the Ascoli theorem which states that if $X$ is a $k$-space, then every compact subset of $\CC(X)$ is evenly continuous, see Theorem 3.4.20 in \cite{Eng}. In \cite{Noble}, Noble proved that every $k_\IR$-space satisfies the conclusion of the Ascoli theorem.
So it is natural to consider the class of spaces which satisfy the conclusion of Ascoli's theorem. Following \cite{BG}, a  space $X$ is called an {\em Ascoli space} if every compact subset $\KK$ of $\CC(X)$  is evenly continuous, that is the map $X\times\KK \ni(x,f)\mapsto f(x)$ is continuous. In other words, $X$ is Ascoli if and only if the compact-open topology of $\CC(X)$ is Ascoli in the sense of \cite[p.45]{mcoy}. One can easily show that $X$ is Ascoli if and only if every compact subset of $\CC(X)$ is equicontinuous. Recall that a subset $H$ of $C(X)$ is {\em equicontinuous} if for every $x\in X$ and each $\e>0$ there is an open neighborhood $U$ of $x$ such that $|f(x')-f(x)|<\e$ for all $x'\in U$ and $f\in H$. For numerous result concerning the Ascoli property, see \cite{Gabr-LCS-Ascoli,Gabr-B1,Gab-LF,Gabr:weak-bar-L(X),Gabr-seq-Ascoli,GGKZ,GGKZ-2}.
The necessity in the next theorem is proved in \cite{GGKZ} and the sufficiency in \cite{Gabr-B1} (for  more general assertions, see Section 5 of \cite{Gabr-seq-Ascoli}).
\begin{theorem} \label{t:Cp-Ascoli}
For a space $X$, the space $C_p(X)$ is Ascoli if, and only if, $X$ has the property $(\kappa)$.
\end{theorem}
\smallskip

Analogously to $k_\IR$-spaces, one can define $s_\IR$-spaces as follows. A space $X$ is called  an {\em $s_\IR$-space} if every sequentially continuous function $f:X\to\IR$ is continuous. Recall that a function $f:X\to Y$ between topological spaces $X$ and $Y$ is called {\em sequentially continuous} or {\em $s$-continuous} if the restriction of $f$ onto each convergent sequence in $X$ is continuous. $s$-continuous functions and $s$-continuous functionals appear naturally not only in General Topology, but also in Functional Analysis, see for example Mazur \cite{Mazur1952} and Wilansky \cite{Wilansky}.
\smallskip

The aforementioned results and discussion motivate us to study $k_\IR$-spaces and $s_\IR$-spaces, and relationships between these and other classes  of spaces, in particular, for function spaces.
\smallskip

Now we describe the content of the article. In Section \ref{sec:preliminary} we recall other notions used in the article. In Section \ref{sec:px}, the basic properties of universally $\sigma$-measurable spaces and universal $\sigma$-measure zero spaces are established. Section \ref{sec:noble} contains the results from N. Noble's article \cite{Nob} and some of their consequences which are actively used in the article. In Section \ref{sec:sfg} we consider spaces which are (strongly) functionally generated by some cover. Numerous characterizations of $k_\IR$-spaces are given in Section \ref{sec:k-R-spaces}. In Section \ref{sec:s-R-spaces}, we propose analogous characterizations of $s_\IR$-spaces. In Proposition \ref{p:sR-R-tightness} we show that each $s_\IR$-space has countable $\IR$-tightness. In Section \ref{p:cp-sep} we study spaces, in particular, $C_p$-spaces which are strongly sequentially separable. In Theorem \ref{t:separable-p-ks} we show that any separable space $X$ such that $\chi(X)<\mathfrak{p}$ is an $s_\IR$-space.
It is natural to study space on which any $s$-continuous function is also $k$-continuous. We call such spaces by $sk$-spaces and study them in Section \ref{sec:sk}. In Section \ref{sec:func-cpb} we consider probability measures on $C_p(X,[-1,1])$ and introduce $\kappa$-measurable and universal $\kappa$-measure zero spaces. These results essentially used in the  last Section \ref{sec:Cp-kR-sR} which is devoted to the $C_p$-part of the following general problem.
\begin{problem} \label{prob:Cp-kR-sR}
Characterize spaces $X$ for which $C_p(X)$ and $\CC(X)$ are $k_\IR$-spaces or $s_\IR$-spaces.
\end{problem}
In Proposition \ref{p:Cp-kR-necessary} we stand necessary conditions on $X$ for which $C_p(X)$ is a $k_\IR$-space. In Theorem \ref{t:exa-kR-Cp}, we prove that $C_p(X)$ is a $k_\IR$-space for any space $X$ with one non-isolated point. In Theorem \ref{t:sss-sr} and Theorem \ref{t:Cp-sR-space-necessary} we give a sufficient and necessary conditions on $X$ for which $C_p(X)$ is an $s_\IR$-space. Finally, under $(CH)$, we show in Proposition \ref{p:Cp-Sierpinski} that there exists a separable metrizable space $X$ such that $C_p(X)$ is an Ascoli space but not a $k_\IR$-space. This gives a partial positive answer to a question posed in \cite{BG}. Numerous open problems are posed.


\section{Preliminary results} \label{sec:preliminary}


Set $\w=\{0,1,2,\dots\}$ and $\NN=\{1,2,\dots\}$. Denote by $\mathbf{s}=\{\tfrac{1}{n}\}_{n\in\NN}\cup\{0\}$ the convergent sequence with the topology induced from $\IR$.

All spaces are assumed to be Tychonoff. Let $X$ be a space. We shall say that $X$ is {\em closely embedded into } a space $Y$ if there is an embedding $f:X\to Y$ whose image $f(X)$ is a closed subspace of $Y$. A base of the compact-open topology of $\CC(X)$ (resp., a base of the pointwise topology of $C_p(X)$) consists of the sets
\[
[K;\e]=\big\{f\in C(X):f(K)\subseteq (\-\e,\e)\big\},
\]
where $K$ is a compact (resp., finite) subset of $X$ and $\e>0$. We denote by $\KK(X)$, $\mathcal{S}(X)$ and $\mathcal{C}(X)$ the family of all compact subsets of $X$, the family of all convergent sequences in $X$ with their limit points, and the family of all countable subsets of $X$, respectively.

A family  $\{ A_i\}_{i\in I}$ of subsets of a set $X$ is called {\em point-finite} if for every $x\in X$, the set $\{ i\in I: x\in A_i\}$ is finite. A family  $\{ A_i\}_{i\in I}$ of subsets of a topological space $X$ is said to be {\em strongly point-finite} if for every $i\in I$ there is an open set $U_i$ of $X$  such that $A_i\subseteq U_i$ and  the family $\{ U_i: i\in I\}$ is point-finite.
Following Sakai \cite{Sak2}, a topological space $X$ is said to have the {\em property $(\kappa)$} if every pairwise disjoint sequence of finite subsets of $X$ has a strongly point-finite subsequence. Clearly, if $Y$ is a subspace of $X$, then also $Y$ has the property $(\kappa)$.

A subset $A$ of a space $X$ is called {\em functionally bounded} if $f(A)$ is a bounded subset of $\IR$ for every $f\in C(X)$.
A space $X$ is called a {\em $\mu$-space} if every functionally bounded subset of $X$ has compact closure. We denote by $\beta X$, $\upsilon X$ and  $\mathcal{D} X$ the Stone-\v{C}ech compactification, the Hewitt realcompactification and the Dieudonn\'e completion of $X$, respectively. We denote by $\mu X$ the {\em $\mu$-completion} of $X$, i.e. $\mu X$  is the smallest $\mu$-subspace of $\beta X$ that contains $X$.

Let  $f:X\to Y$ be a continuous mapping between spaces $X$ and $Y$. The {\em adjoint map} $f^{\#}:C_p(Y)\to C_p(X)$ of $f$ is defined by $f^{\#}(g)(x):=g\big(f(x)\big)$.

A space $X$ is {\em perfectly normal} if  it is normal and every closed subset is a $G_\delta$ subset. It is well-known that a space $X$ is perfectly normal if, and only if, any closed subset $A$ of $X$ is a zero set (that is, $A=f^{-1}(0)$ for some $f\in C(X)$). It is known that any Souslin space is perfectly normal. Recall that  a space $X$ is {\em Souslin} if $X$ is a continuous image of a Polish space (a space is {\em Polish} if it is separable and complete-metrizable).

Following Karnik and Willard \cite{Kar-Wil}, a continuous mapping $p:X\to Y$ is called {\em $R$-quotient} ({\em real-quotient}) if $p$ is continuous, onto, and every function $\phi:Y\to\IR$ is continuous whenever the composition $\phi\circ p$ is continuous. Clearly, every quotient mapping as well as each surjective open mapping is $R$-quotient. The following characterization of $R$-quotient mappings is Proposition 0.4.10 of \cite{Arhangel}.
\begin{proposition}[\cite{Arhangel}] \label{p:R-quotient-characterization}
A continuous surjective mapping $p:X\to Y$ is $R$-quotient if, and only if, $p^{\#}\big(C_p(Y)\big)$ is a closed subspace of $C_p(X)$.
\end{proposition}

We shall use repeatedly the following assertion, see Theorem 4.7 of \cite{Oku-Par}.
\begin{theorem}[\cite{Oku-Par}] \label{t:rq-prod}
Let $p:X\to Y$ be an $R$-quotient mapping between spaces. Then for every locally compact space $Z$, the product mapping $p\times\Id_Z:X\times Z\to Y\times Z$ is $R$-quotient.
\end{theorem}

Now we recall other classes of topological spaces used in the article.
Let $X$ be a space. Recall that a subspace $Y$ of $X$ is {\em sequentially dense} if for each point $x\in X$ there is a sequence $\{y_n\}_{n\in\w}\subseteq Y$ converging to $x$. The space $X$ is called
\begin{enumerate}
\item[$\bullet$] {\em Fr\'{e}chet--Urysohn} if each subspace $Y$ of $X$ is sequentially closed in its closure $\overline{Y}$;
\item[$\bullet$] {\em sequential} if for each non-closed subset $A\subseteq X$ there is a sequence $\{a_n\}_{n\in\NN}\subseteq A$ converging to some point $a\in \bar A\setminus A$;
\item[$\bullet$] {\em $\kappa$-Fr\'{e}chet--Urysohn} if for every open subset $U$ of $X$ and every $x\in \overline{U}$, there exists a sequence $\{x_n\}_{n\in\w} \subseteq U$ converging to $x$;
\item[$\bullet$] a {\em sequentially Ascoli space} if every convergent sequence in $\CC(X)$ is equicontinuous.
\end{enumerate}
The notion of  $\kappa$-Fr\'{e}chet--Urysohn spaces was introduced by Arhangel'skii (see for example \cite{Gabr-B1}), and sequentially Ascoli spaces were defined in \cite{Gabr:weak-bar-L(X)}.
The relationships between the aforementioned notions are described in the next diagram
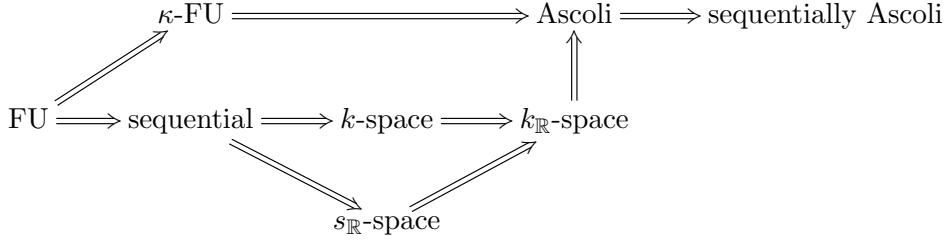
\[
\xymatrix{
& \mbox{$\kappa$-FU} \ar@{=>}[rr]&& \mbox{Ascoli}  \ar@{=>}[r] & \mbox{sequentially Ascoli}  \\
\mbox{FU} \ar@{=>}[r]\ar@{=>}[ru] & \mbox{sequential} \ar@{=>}[r]\ar@{=>}[rd] &  \mbox{$k$-space} \ar@{=>}[r] &  \mbox{$k_\IR$-space} \ar@{=>}[u]&\\
&& \mbox{$s_\IR$-space} \ar@{=>}[ru] && }
\]
Note that, by the main result of \cite{GGKZ-2}, the space $C_p(\w_1)$ is $\kappa$-Fr\'{e}chet--Urysohn  which is not a $k_\IR$-space.



\section{Universally $\si$-measurable and universal $\si$-measure zero spaces} \label{sec:px}


For a compact space $K$, we denote  by $C(K)$ the Banach lattice of all continuous real-valued functions on $K$ endowed with the $\sup$-norm, and let
\[
P(K) = \{\mu\in C(K)': \|\mu\|=1, \mu\geq  0\}
\]
be the positive unit sphere of the dual Banach space $C(K)'$, endowed with the weak$^\ast$ topology. By the fundamental Riesz Theorem, the functionals $\mu \in P (K)$ can be identified with Borel regular probability measures on $K$, see \cite[1.2]{Fedorchuk1991}. Every continuous function $f: X \to Y$ between compact spaces induces a map $P (f): P(X) \to P(Y)$ defined by
\[
P(f)(\mu)(g) := \mu(g \circ f )\; \mbox{ for }\; g \in C(Y).
\]

Let $X$ be a space. A countably additive measure $\mu$ on the $\si$-algebra $\cB(X)$ of Borel subsets of $X$ is called a {\em Borel} measure.
A measure $\mu$ is called {\em continuous} if $\mu(\{x\})=0$ for every $x\in X$.
A measure $\mu$ is called {\em discrete} if $|\mu|(C)=1$ for some countable $C\subseteq X$. For $x\in X$, we denote by $\delta_x$ the Dirac measure at $x$.
The set of all probability discrete measures on $X$ is denoted by $P_d(X)$. Each $\mu\in P_d(X)$ can be uniquely represented in the form $\mu=\sum_{n=1}^\infty \lambda_n \delta_{x_n}$, where $x_n\in X$ are distinct, $\lambda_n> 0$ and $\sum_{n=1}^\infty \lambda_n=1$.

We denote by $P_b(X)$ the family of all probability Borel $\sigma$-additive measures on a space $X$. A measure $\mu\in P_b(X)$ is called {\em Radon} if for every Borel set $B\subset X$ and for every $\e>0$, there exists a compact $K\subseteq B$ such that $\mu(B\setminus K)<\e$. The set of all probability Radon measures on $X$ is denoted by $P(X)$. A measure $\mu$ is called {\em $\tau$-additive} if
\[
\mu\big(\bigcup\cU\big) = \sup \big\{ \mu\big(\bigcup\cV\big) : \cV\subset \cU,\ |\cV|<\w \big\}
\]
for any family $\cU$ of open subsets of $X$. The set of all probability $\tau$-additive measures on $X$ is denoted by $P_\tau(X)$.


Following \cite{BanakhChigogidzeFedorchuk2003}, we consider the aforementioned sets of probability measures as subspaces of $P(\beta X)$:
\begin{enumerate}
\item[$\bullet$] $P_d(X)=\{\mu\in P(\beta X): \mu(A)=1$ for some $A\subseteq X, |A|\leq\om \}$;
\item[$\bullet$] $P(X)=\{\mu\in P(\beta X): \mu(A)=1$ for some $\si$-compact $A\subseteq X\}$;
\item[$\bullet$] $P_\tau(X)=\{ \mu\in P(\beta X): \mu(K)=0 $ for every compact $K\subseteq \bt X\setminus X\}$;
\item[$\bullet$] $P_\si(X)=\{ \mu\in P(\beta X): \mu(K)=0 $ for every compact $K\subseteq \bt X\setminus X$ of  $G_\delta$ type$\}$.
\end{enumerate}
Then
\[
P_d(X) \subseteq P(X)  \subseteq P_\tau(X)  \subseteq P_\sigma(X)  \subseteq P(\beta X).
\]
By the general Riesz theorem,  the elements of $P(\beta X)$ can be identified with finitely additive regular probability measures defined on the algebra of Baire subsets of $X$ (i.e., the smallest algebra generated by all functionally closed subsets of $X$). Under this identification, the elements of $P_\sigma(X)$ correspond to $\sigma$-additive probability measures  defined
on Baire subsets of $X$ (see  \cite[Section 1]{Fedorchuk1991}).


\begin{proposition} \label{p:measure-pns}
If $X$ is a perfectly normal space, then $P_b(X)=P_\sigma(X)$.
\end{proposition}
\begin{proof}
Since in a perfectly normal space every closed set is a zero-set, the algebra of Baire subsets of $X$ coincides with the algebra of Borel subsets of $X$.
Thus, $P_b(X)=P_\sigma(X)$.
\end{proof}

Let $X$ be a subspace of a space $Y$. The family $P_b(X)$ can be naturally considered as a subspace of $P_b(Y)$: a measure $\mu\in P_b(X)$ corresponds to the measure $\tilde\mu\in P_b(Y)$ such that $\tilde\mu(B)=\mu(B\cap Y)$ for $B\in\cB(Y)$. We shall always identify $P_b(X)$ with its image in $P_b(Y)$ and write $P_b(X)\subseteq P_b(Y)$.

Let $A$ be a subset of a space $X$, and let $\mu\in P_b(X)$. The {\em external} and the {\em internal} measure of $A$ are defined as follows
\begin{align*}
\mu^\ast(A) &= \inf\{\mu(B): A\subseteq B\in \cB(X)\},
&
\mu_\ast(A) &= \sup\{\mu(B): A\supseteq B\in \cB(X)\}.
\end{align*}
A subset $M$ of $X$ is {\em $\mu$-measurable} if $\mu_\ast(M)=\mu^\ast(M)$. A subset $M$ of $X$ is called {\em universally measurable} if $M$ is $\mu$-measurable for any $\mu\in P_\sigma(X)$. A subset $A$ is called {\em universal measure zero} if $\mu^\ast(A)=0$ for any continuous $\mu\in P_b(X)$.




Let $X$ and $Y$ be spaces, and let  $f: X \to Y$  be a continuous mapping. Denote by $\beta f: \beta X\to \beta Y$ the unique extension of $f$. Then, by \cite{BanakhChigogidzeFedorchuk2003}, we have the following inclusions
\begin{equation} \label{equ:measure-1}
\begin{aligned}
P(\beta f)\big(P_d(X)\big)\subseteq P_d(Y),& \quad P(\beta f)\big(P(X)\big)\subseteq P(Y), \\
P(\beta f)\big(P_\tau(X)\big)\subseteq P_\tau(Y),& \quad P(\beta f)\big(P_\sigma(X)\big)\subseteq P_\sigma(Y).
\end{aligned}
\end{equation}

We shall say that a space $X$ is {\em universally $\sigma$-measurable} if $P_\sigma(X)=P(X)$. Such spaces in \cite{BanakhChigogidzeFedorchuk2003} were called universally measurable. The last inclusion in (\ref{equ:measure-1}) and the equality $P_\sigma(Y)=P(Y)$ imply the next assertion.

\begin{proposition}\label{p:px:ci}
Let $X$ be a space, $Y$ be a universally $\sigma$-measurable space, and let $f: X \to Y$ be a continuous map. If $\mu\in P_\sigma(X)$, then $P(\beta f)(\mu)\in P(Y)$.
\end{proposition}

Let us call the space $X$ to be {\em universal $\si$-measure zero} if $P_\sigma(X)=P_d(X)$. Clearly, any universal $\si$-measure zero space is universally $\si$-measurable. By Theorem 1 of \cite{Knowles1967}, all Radon measures on a compact space $X$ are discrete if, and only if, $X$ is scattered. Therefore, the following statement holds.

\begin{proposition}\label{p:px:sccs}
A compact space $X$ is universal $\sigma$-measure zero if, and only if, $X$ is scattered.
\end{proposition}

Now we characterize universal $\sigma$-measure zero spaces.
\begin{theorem}\label{t:px:sccs}
A space $X$ is universal $\sigma$-measure zero if, and only if, $X$ is universally $\sigma$-measurable and any compact subspace of $X$ is scattered.
\end{theorem}

\begin{proof}
$(\Ra)$ Since $X$ is universal $\sigma$-measure zero, the space $X$ is universally $\sigma$-measurable. Let $K\subseteq X$ be compact. Since $P(K)\subseteq P(X)=P_d(X)$, then $P(K)=P_d(K)$ and $K$ is universal $\sigma$-measure zero. It follows from Proposition \ref{p:px:sccs} that $K$ is scattered.

$(\La)$ Let $\mu\in P_\sigma(X)$. Since $X$ is universally $\sigma$-measurable, we have $\mu\in P(X)$. Then $\mu(S)=1$ for some $\sigma$-compact $S\subseteq X$. Let $S=\bigcup_{i=1}^\infty S_n$, where each $S_n$ is compact. Since $S_n$ is scattered, then, by Proposition \ref{p:px:sccs}, $S_n$ is universal $\si$-measure zero. Therefore, the measure $\mu{\restriction}_{S_n}$ is discrete. Then the measure $\mu$ is discrete.
\end{proof}

Recall (see \cite{Schwartz1973}) that a space $X$ is called {\em Radon} if $P_b(X)=P(X)$. Following \cite[439B]{Fremlin-IV}, a space $X$ is called {\em universally negligible} if $P_b(X)=P_d(X)$.

\begin{proposition}\label{p:px:2}
Let $X$ be a perfectly normal space. Then
\begin{enumerate}
\item[{\rm(i)}] $X$ is a universally $\sigma$-measurable space if, and only if, $X$ is Radon;
\item[{\rm(ii)}] $X$ is a universal $\sigma$-measure zero space if, and only if, $X$ is universally negligible.
\end{enumerate}
\end{proposition}

\begin{proof}
(i) follows from Proposition \ref{p:measure-pns}.

(ii) follows from (i), Theorem \ref{t:px:sccs}, and the following fact \cite[439C]{Fremlin-IV}:  a space $X$ is universally negligible if, and only if, $X$ is Radon and any compact subspace of $X$ is scattered.
\end{proof}

By Theorem 10 of   \cite[p.~122]{Schwartz1973}, every Souslin space is Radon. This result and  Proposition \ref{p:px:2} imply the following assertion.

\begin{proposition}\label{p:px:3}
Any Souslin space is universally $\sigma$-measurable.
\end{proposition}

\begin{theorem}\label{t:px:polish}
Let $M$ be a subspace of  a Polish space $X$. Then:
\begin{enumerate}
\item[{\rm(i)}] $M$ is a universally measurable set in $X$ if, and only if, $M$ is a universally $\sigma$-measurable space;
\item[{\rm(ii)}] $M$ is a universal measure zero set in $X$ if, and only if, $M$ is a universal $\sigma$-measure zero space.
\end{enumerate}
\end{theorem}

\begin{proof}
Recall that, by Theorem 10 of   \cite[p.~122]{Schwartz1973}, each Polish space is Radon.

(i) The set $M$ is  universally measurable in $X$ if, and only if, $M$ is a Radon space (see Propositions 8 and 9 in \cite[p.~118]{Schwartz1973}). Now Proposition \ref{p:px:2}(i) applies.

(ii) From (i) and Proposition \ref{p:px:2} it follows that $M$ is a universal $\sigma$-measure zero space if, and only if, $X$ is a universally $\sigma$-measurable set in $X$ and any compact subspace of $M$ is at most countable being scattered. By 1.2.1. Characterization of  \cite{Nishiura2008}, the latter property is equivalent to $M$ of being a universal measure zero set in $X$.
\end{proof}

\begin{theorem}\label{t:px:ums-realcompact}
If $X$ a universally $\sigma$-measurable space, then $P(X)$ and $X$ are realcompact spaces.
\end{theorem}
\begin{proof}
By Theorem 1.1 of  \cite{BanakhChigogidzeFedorchuk2003}, the space $P_\sigma(X)$ is always realcompact. Since $P(X)=P_\sigma(X)$, we obtain that $P(X)$ is also realcompact. Since $X$ is closely embedded in $P(X)$, we obtain that $X$ is realcompact, too.
\end{proof}


\section{N.~Noble's results on infinite products of spaces} \label{sec:noble}


In this section we formulate some results and their consequences from Noble's article \cite{Nob} that are necessary in what follows.
In \cite{Nob} he considered $\Sigma^0$-products, which are now called $\sigma$-products. Accordingly, $\Si^0$-continuous mappings will be called $\si$-continuous mappings.

Let $X=\prod_{\alpha\in \AAA}X_\alpha$ be a product of spaces. For $x=(x_\alpha)_\alpha, y=(y_\alpha)_\alpha\in X$, let
\[
\begin{aligned}
\delta(x,y) &= \{ \alpha\in \AAA : x_\alpha\neq y_\alpha\},\\
\sigma(X,x) &= \{ y\in X : |\delta(x,y)|<\w\}.
\end{aligned}
\]
Sets of the form $\sigma(X,x)$ are called  {\em $\sigma$-products}. A mapping $f$ from $X$ to a space $Y$ is called {\em $\sigma$-continuous} if $f$ is continuous on every $\sigma$-product in $X$; $f$ is called {\em $2$-continuous} if $f$ is continuous on every subset of $X$ of the form $\prod_{\alpha\in \AAA} Y_\alpha$, where $|Y_\alpha|\leq 2$. The following result was proved in Theorem 1.1 of Noble \cite{Nob}.
\begin{theorem}[\cite{Nob}] \label{t:Noble-continuity}
Let $X=\prod_{\alpha\in \AAA}X_\alpha$ be a product of spaces, and let $Y$ be a space. If the mapping $f:X\to Y$ is $\sigma$-continuous and $2$-continuous, then $f$ is continuous.
\end{theorem}

An ultrafilter on a set $A$ is $\tau$-complete if it is closed under intersections of families of cardinality $<\tau$. Let us say that a cardinal $|A|$ is {\em Ulam-measurable} if there exists an $\om_1$-complete free ultrafilter on $A$. (A free ultrafilter is the same as a nonprincipal ultrafilter.) A cardinal is Ulam-measurable if and only if it greater than or equal to the first measurable cardinal.
As usual \cite[Definition 10.3]{Jech2006},
an uncountable cardinal $\tau$ is measurable if there exists a $\tau$-complete free ultrafilter on $\tau$.

In \cite{Nob} Ulam-measurable cardinals are called measurable cardinals.

Denote by $\mathbf{2}$ the discrete two-point space $\{0,1\}$. Recall that a cardinal $\tau$ is {\em sequential} if there exists a sequentially continuous function $f: \mathbf{2}^\tau\to \IR$ which is discontinuous. Any Ulam-measurable cardinal is sequential, see  Keisler and  Tarski \cite{KeislerTarski1963}.
On the other hand, Mazur showed in \cite{Mazur1952} that each cardinal less than the first weakly inaccessible cardinal is not sequential. In particular, $\aleph_1$ is not sequential.

In \cite{KeislerTarski1963}, Keisler and  Tarski posed the following problem: {\em Is it true that a cardinal $\tau$ is sequential if, and only if, it is Ulam-measurable}? Under the Martin's Axiom $\mathrm{MA}$, a cardinal $\tau$ is sequential if, and only if,  it is Ulam-measurable \cite{Solovay1971,Chudnovskij1973}.
If the continuum $\mathfrak{c}$ is not sequential, then every sequential cardinal is a Ulam-measurable cardinal \cite{BalcarHusek2001}.

\begin{proposition} \label{p:noble:prod}
Let $X=\prod_{\alpha\in \AAA}X_\alpha$ be a product of spaces.
\begin{enumerate}
\item[{\rm(i)}]
If all $\sigma$-products in $X$ are $k_\IR$-spaces, then also $X$ is a $k_\IR$-space.
\item[{\rm(ii)}]
If all $\sigma$-products in $X$ are $s_\IR$-spaces and $|\AAA|$ is not sequential, then also $X$ is an $s_\IR$-space.
\end{enumerate}
\end{proposition}
\begin{proof}
Let $f\: X\to\IR$ be a function.

(i) Assume that $f$ is $k$-continuous. Since all $\sigma$-products in $X$ are $k_\IR$-spaces, the function $f$ is $\sigma$-continuous. Since $f$ is $k$-continuous, it is $2$-continuous.
It follows from Theorem \ref{t:Noble-continuity} that $f$ is continuous.

(ii) Assume that $f$ is $s$-continuous. Since all $\sigma$-products in $X$ are $s_\IR$-spaces, $f$ is $\sigma$-continuous. Since $f$ is $s$-continuous and $|\AAA|$ is not sequential, the function $f$ is $2$-continuous. It follows from Theorem \ref{t:Noble-continuity} that $f$ is continuous.
\end{proof}

Theorem 3.2(ii) from \cite{Nob} implies the following statement which shows that the condition in (ii) of Proposition \ref{p:noble:prod} that $|\AAA|$ is not sequential is necessary.

\begin{proposition}[\cite{Nob}]\label{p:product-sR--seq-card}
Let $X=\prod_{\alpha\in\AAA}$ be the product of a nonempty family $\{X_\alpha\}_{\alpha\in\AAA}$ of spaces such that $|X_\alpha|\geq 2$ for every $\alpha\in\AAA$. If $X$ is an $s_\IR$-space, then $|\AAA|$ is not sequential.
\end{proposition}



Since each $\sigma$-product of first countable spaces is Fr\'{e}chet--Urysohn \cite{Nob}, the following assertion follows from Proposition \ref{p:noble:prod}.

\begin{proposition}[\cite{Nob}] \label{p:noble:fc-prod}
Let $X=\prod_{\alpha\in \AAA}X_\alpha$ be a product of first countable spaces.
\begin{enumerate}
\item[{\rm(i)}]
The space $X$ is a $k_\IR$-space.
Consequently, for any cardinal $\lambda$, the product $\IR^\lambda$ is a $k_\IR$-space.
\item[{\rm(ii)}]
If $|\AAA|$ is not sequential, then $X$ is an $s_\IR$-space.
Consequently, the product $\IR^{\aleph_n}$ is an $s_\IR$-space for every $n\in\NN$.
\end{enumerate}
\end{proposition}

A topological space is of {\em pointwise countable type} (or {\em point-countable type}) provided every  point in the space is contained in a compact set of countable character.
Any \v Chech complete, first countable or locally compact space is of pointwise countable type \cite[3.1.E(b), 3.3.I]{Eng}.

\begin{proposition}[\cite{Nob}] \label{p:noble:pwct}
Let $X=\prod_{\alpha\in \AAA}X_\alpha$ be a product of spaces of pointwise countable type (in particular, locally compact spaces). Then the space $X$ is a $k_\IR$-space.
\end{proposition}
\begin{proof}
Let $f\: X\to\IR$ be a $k$-continuous function. Then, by  Theorem 2.4 from \cite{Nob}, $f$ is $\si$-continuous. Since $f$ is $k$-continuous, then $f$ is $2$-continuous.
It follows from Theorem \ref{t:Noble-continuity} that $f$ is continuous.
\end{proof}

\begin{proposition}[\cite{Nob}] \label{p:noble:pseudo-prod}
Let $X=\prod_{\alpha\in \AAA}X_\alpha$ be a product of pseudocompact spaces.
\begin{enumerate}
\item[{\rm(i)}]
If each $X_\alpha$ is a \kr-space, then $X$ is a pseudocompact $k_\IR$-space.
\item[{\rm(ii)}]
If each $X_\alpha$ is a \sr-space and $|\AAA|$ is not sequential, then $X$ is an pseudocompact $s_\IR$-space.
\end{enumerate}
\end{proposition}
\begin{proof}
Part (i) follows from Theorem 4.2 and Theorem 5.6 of \cite{Nob}.

(ii) From (i) it follows that $X$ is pseudocompact.
From \cite[Theorem 5.3]{Nob} it follows that $X$ is a \sr-space.
\end{proof}

\begin{remark} \label{rem:kR-perman}
(i) The conclusion of Proposition \ref{p:noble:fc-prod} is the best possible in the following sense. If to weaken the first countability to the property of being a Fr\'{e}chet--Urysohn space, then the product can be not a  $k_\IR$-space. In fact,  Proposition 2.14 of \cite{Gabr-seq-Ascoli} states that there is a  Fr\'{e}chet--Urysohn space $X$ whose square $X\times X$ is not even an Ascoli space. On the other hand, by (the proof of) Theorem 1.2 of \cite{Gab-LF}, also the product $\ell_1\times \varphi$ is  not an Ascoli space,  where $\varphi=\IR^{(\w)}$ is the direct locally convex sum of a countable family of $\IR$-s, and, by Nyikos \cite{nyikos},  $\varphi$ is a sequential, non-Fr\'{e}chet--Urysohn space.

(ii) It is well-known that a closed subspace of a $k$-space is also a $k$-space. However, an analogous result for $k_\IR$-spaces is not true in general. Indeed, let $X$ be a countable non-discrete space whose compact sets are finite. Then, by Corollary 2.3 of \cite{Gabr-lc-Ck}, $X$ is not an Ascoli space. On the other hand, being Lindel\"{o}f the space $X$ is realcompact (\cite[Theorem~3.11.12]{Eng}), and hence it is closely embedded into some power $\IR^\lambda$ (\cite[Theorem~3.11.3]{Eng}) which is a $k_\IR$-space (and a $\kappa$-Fr\'{e}chet--Urysohn space).
If we assume additionally that an Ascoli space (resp., a $k_\IR$-space) $X$ is a stratifiable space, then, by Proposition 5.11 of \cite{BG}, any closed subspace of $X$ also is Ascoli (resp., a $k_\IR$-space).

(iii) Arhangel'skii proved (see \cite[3.12.15]{Eng}) that if $X$ is a hereditary $k$-space (i.e., {\em every} subspace of $X$ is a $k$-space), then $X$ is Fr\'{e}chet--Urysohn. This result was essentially strengthen in Theorem 2.21 of \cite{Gabr-seq-Ascoli} where it is  shown that a hereditary Ascoli space must be Fr\'{e}chet--Urysohn. Consequently, a  hereditary $k_\IR$-space is Fr\'{e}chet--Urysohn.
\end{remark}


\section{Spaces which are (strongly) functionally generated by a cover} \label{sec:sfg}


The importance of covers which (strongly) functionally generate a space $X$ is explained by numerous important result obtained in Section III, \S~4 of \cite{Arhangel}. Such covers are important also to obtain new characterizations of  $k_\IR$-spaces and  $s_\IR$-spaces.

Let $\mathcal{M}$ be a cover of a space $X$. Recall that  $X$ is called
\begin{enumerate}
\item[$\bullet$]  {\em generated} by the family $\mathcal{M}$ if for each non-closed set $A\subseteq X$, there is an $M \in\mathcal{M}$ such that $M\cap A$ is not closed in $M$;
\item[$\bullet$] {\em strongly functionally generated} by the family $\mathcal{M}$ if for each real-valued discontinuous function $f$ on $X$, there is an $M \in\mathcal{M}$ such that the function $f{\restriction}_M : M \to \IR$ is discontinuous;
\item[$\bullet$]     {\em functionally generated} by the family $\mathcal{M}$ if for every discontinuous function $f: X \to \IR$, there is an $M \in\mathcal{M}$ such that the function  $f{\restriction}_M$ cannot be extended to a real-valued continuous function on the whole space $X$.
\end{enumerate}
It is clear that $X$ is a $k$-space (resp., a sequential space or $X$ has countable tightness) if it is generated by the family $\KK(X)$ (resp., by $\mathcal{S}(X)$ or by $\mathcal{C}(X)$). The space $X$ has countable $\IR$-tightness if, and only if, it is functionally generated by $\mathcal{C}(X)$.

Below we characterize spaces generated by a cover, from which we deduce in Proposition \ref{p:generated-sfg} that any such space is strongly functionally generated by that cover.

\begin{proposition} \label{p:generated-family}
Let $\mathcal{M}$ be a cover of a space $X$. Then the following assertions are equivalent:
\begin{enumerate}
\item[{\rm(i)}] $X$ is generated by the family $\mathcal{M}$;
\item[{\rm(ii)}] the natural mapping $I:\bigoplus \mathcal{M}\to X$, defined by $I(x)=x$ for $x\in M\in \mathcal{M}$, is quotient;
\item[{\rm(iii)}] for each non-closed $Q\subseteq X$, there are $x\in\overline{Q}\SM Q$ and $M\in\mathcal{M}$ such that $x\in M\cap \overline{M\cap Q}$.
\end{enumerate}
\end{proposition}

\begin{proof}
(i)$\Ra$(iii) Let $Q\subseteq X$ be non-closed. Then, by (i), there is $M\in\mathcal{M}$ such that $M\cap Q$ is not closed in $M$. So, there is a point $x\in M$ such that $x\in \overline{M\cap Q}\SM (M\cap Q)$; in particular, $x\in \overline{Q}$. It remains to note that $x\not\in Q$ (otherwise, we would have $x\in M\cap Q$ that is impossible).

(iii)$\Ra$(i)  Let $Q$ be a non-closed subset of $X$. By (iii), choose $x\in \overline{Q}\SM Q$ and $M\in \mathcal{M}$ such that $x\in M\cap \overline{M\cap Q}$. Since $x\not\in Q$ and $x\in M$ it follows that $M\cap Q$ is not closed in $M$, as desired.

(i)$\LRa$(ii) A space $X$ is generated by the family $\mathcal{M}$ if, and only if, the following condition is satisfied: a set $F\subseteq X$ is closed if and only if $F\cap M$ is closed in $M$ for each $M\in \cM$. It remains to apply the definitions of the mapping $I$ and the definition of a quotient mapping.
\end{proof}

The implication (i)$\LRa$(iii) of the following statement is Theorem 2(1) of \cite{Delgadilo2000}.
\begin{proposition} \label{p:R-quotient-sfg}
Let $\mathcal{M}$ be a cover of a space $X$. Then the following assertions are equivalent:
\begin{enumerate}
\item[{\rm(i)}] $X$ is strongly functionally generated by the family $\mathcal{M}$;
\item[{\rm(ii)}] the natural mapping $I:\bigoplus \mathcal{M}\to X$  is $R$-quotient; 
\item[{\rm(iii)}]
the adjoint mapping
\[
I^{\#}\: C_p(X)\to C_p\left(\bigoplus\cM\right)=\prod_{M\in \cM} C_p(M)
\]
is an embedding  with closed  image.
\end{enumerate}
\end{proposition}

\begin{proof}
(i)$\Ra$(ii) It is clear that $I$ is continuous and surjective. Now, let $f:X\to\IR$ be a function such that $f\circ I$ is continuous. Assuming that $f$ is discontinuous we could find $M\in\mathcal{M}$ such that $f{\restriction}_M$ is discontinuous, too. Since $f{\restriction}_M=f\circ I{\restriction}_M$ we obtain a contradiction. Thus $f$ is continuous, and hence $I$ is $R$-quotient.

(ii)$\Ra$(i) Let $f:X\to\IR$ be a discontinuous function. Since  $I$ is $R$-quotient, we obtain that $f\circ I$ is also discontinuous. Therefore there is $M\in\mathcal{M}$ such that $f{\restriction}_M=f\circ I{\restriction}_M$ is discontinuous. Thus  $X$ is strongly functionally generated by the family $\mathcal{M}$.

The implication (ii)$\LRa$(iii) follows from Proposition \ref{p:R-quotient-characterization}.
\end{proof}

Propositions \ref{p:generated-family}(ii) and \ref{p:R-quotient-sfg}(ii) imply the following assertion.

\begin{proposition} \label{p:generated-sfg}
If a space $X$ is generated by a family $\mathcal{M}$ of its subsets, then $X$ is strongly functionally generated by the family $\mathcal{M}$.
\end{proposition}

Taking into account that each space $X$, which is  strongly functionally generated by $\mathcal{M}$, is also functionally generated by $\mathcal{M}$, we obtain the following implications
\[
\xymatrix{
\mbox{generated} \ar@{=>}[r]  & \mbox{strongly functionally generated} \ar@{=>}[r] & \mbox{functionally generated.}}
\]

\begin{proposition} \label{p:generated-loc_fin}
Let $\mathcal{M}$ be a locally finite cover of a space $X$ such that every $M\in\mathcal{M}$ is closed. Then $X$ is generated by the family $\mathcal{M}$.
\end{proposition}
\begin{proof}
The natural mapping $I:\bigoplus \mathcal{M}\to X$ is closed and, hence, quotient (see Corollary 2.4.8 of \cite{Eng}). Therefore, by Proposition \ref{p:generated-family}, $X$ is generated by the family $\cM$.
\end{proof}

Recall that a subset $A$ of a space $X$ is {\em $C$-embedded} if each continuous function $f:A\to \IR$ can be extended to a continuous function $\bar f: X\to \IR$.
\begin{proposition} \label{p:fg=sfg-C-embedded}
Let $\mathcal{M}$ be a cover of a space $X$ such that every $M\in\mathcal{M}$ is $C$-embedded in $X$. Then the following assertions are equivalent:
\begin{enumerate}
\item[{\rm(i)}] $X$ is functionally generated by the family $\mathcal{M}$;
\item[{\rm(ii)}] $X$ is strongly functionally generated by the family $\mathcal{M}$.
\end{enumerate}
\end{proposition}

\begin{proof}
We need to prove only the implication (i)$\Ra$(ii). Let $f:X\to \IR$ be a discontinuous function. Choose $M\in\mathcal{M}$ such that the function  $f{\restriction}_M$ cannot be extended to a real-valued continuous function on the space $X$. Since $M$ is $C$-embedded, it follows that $f{\restriction}_M$ is discontinuous, as desired.
\end{proof}


The following notion plays an essential role in what follows.

\begin{definition} \label{def:M-dense}
Let $\mathcal{M}$ be a cover of a space $X$. A subset $A$ of $X$ is called {\em $\mathcal{M}$-dense} if for every $x\in X$, there is $M\in \mathcal{M}$ such that $x\in M\cap \overline{M\cap A}$.
\end{definition}
Observe that each $\mathcal{M}$-dense subset of $X$ is dense.

\begin{proposition} \label{p:M-sfg}
Let $\mathcal{M}_0$ and $\mathcal{M}$ be covers of a space $X$ such that $\mathcal{M}_0\subseteq \mathcal{M}$. Assume that each $M\in\mathcal{M}_0$ is $\mathcal{M}$-dense in $X$. Then $X$ is strongly functionally generated by the family $\mathcal{M}$.
\end{proposition}

\begin{proof}
Let $f:X\to\IR$ be a discontinuous function at some point  $x\in X$. Since $\mathcal{M}_0$ is a cover of $X$, there is an $M\in \mathcal{M}_0\subseteq \mathcal{M}$ such that $x\in M$. If $f{\restriction}_M$ is discontinuous we are done. Assume that $f{\restriction}_M$ is continuous. As $f$ is discontinuous at $x$, there is $\e>0$ such that
\[
x\in \overline{Q}\SM Q, \;\mbox{ where }\; Q=X\SM f^{-1} \big( (f(x)-\e,f(x)+\e)\big).
\]
Since $f$ is continuous on $M$, there is a neighborhood $U$ of $x$ such that
\[
f(M\cap U) \subseteq \big( f(x)-\tfrac{\e}{2},f(x)+\tfrac{\e}{2}\big).
\]
Let $y\in Q\cap U$. Since $M$ is $\mathcal{M}$-dense, there exists an $M'\in \mathcal{M}$ such that $y\in M'\cap \overline{M'\cap M}$; in particular, $y\in \overline{M'\cap M\cap U}$. It remains to note that the function $f{\restriction}_{M'}$ is discontinuous at $y$ because $f(y)\not\in (f(x)-\e,f(x)+\e)$ (by the definition of $Q$), but $f(M'\cap M\cap U)\subseteq \big( f(x)-\tfrac{\e}{2},f(x)+\tfrac{\e}{2}\big)$.
\end{proof}

\begin{proposition} \label{p:homeo-sfg}
Let $X$ be a homogenous space, and let $\mathcal{M}$ be a cover of $X$ such that $h(M)\in \mathcal{M}$ for any $M\in\mathcal{M}$ and each homeomorphism $h:X\to X$. If there is an $\mathcal{M}$-dense subset $A$ of $X$, then $X$ is strongly functionally generated by the family $\mathcal{M}$.
\end{proposition}

\begin{proof}
Let $\mathcal{M}_0:=\big\{ h(A):  \mbox{ where } h:X\to X \mbox{ is a homeomorphism}\big\}$. Then $\mathcal{M}_0$ is a cover of $X$ containing $\mathcal{M}$-dense subsets of $X$. Thus, by Proposition \ref{p:M-sfg}, $X$ is strongly functionally generated by  $\mathcal{M}$.
\end{proof}


\section{Characterizations of $k_\IR$-spaces} \label{sec:k-R-spaces}


In this section we give two types of characterizations of $k_\IR$-spaces. In the first one we use the notions of $R$-quotient mappings and (strongly) functionally generated families (which are closely related by Propositions \ref{p:R-quotient-sfg} and \ref{p:fg=sfg-C-embedded}).

\begin{theorem} \label{t:characterization-kR}
For a Tychonoff space $X$, the following assertions are equivalent:
\begin{enumerate}
\item[{\rm(i)}] $X$ is a $k_\IR$-space;
\item[{\rm(ii)}]  $X$ is strongly functionally generated by the family $\KK(X)$;
\item[{\rm(iii)}]  $X$ is functionally generated by the family $\KK(X)$;
\item[{\rm(iv)}] $X$ is strongly functionally generated by a family of $k_\IR$-subspaces of $X$;
\item[{\rm(v)}]  $X$ is an $R$-quotient image of some topological sum of compact spaces;
\item[{\rm(vi)}]  $X$ is an $R$-quotient image  of a locally compact space;
\item[{\rm(vii)}]  $X$ is an $R$-quotient  image of a $k$-space;
\item[{\rm(viii)}]  $X$ is an $R$-quotient  image of a $k_\IR$-space.
\end{enumerate}
\end{theorem}

\begin{proof}
The equivalence (i)$\Leftrightarrow$(ii) is essentially the definition of $k_\IR$-spaces.

(ii)$\Leftrightarrow$(iii) follows from Proposition \ref{p:fg=sfg-C-embedded} because each compact subset of $X$ is $C$-embedded in $X$.

(ii)$\Ra$(iv) is trivial, and (ii)$\Ra$(v) follows from Proposition \ref{p:R-quotient-sfg}.

(iv)$\Ra$(viii) Let $\cM$ be a family of $k_\IR$-subspaces of $X$ such that $X$ is strongly functionally generated by $\mathcal{M}$. It follows from Proposition \ref{p:R-quotient-sfg} that $X$ is an $R$-quotient image of the discrete sum of $k_\IR$-spaces $\bigoplus\cM$. It remains to note that  $\bigoplus\cM$ is a $k_\IR$-space.

The implications (v)$\Ra$(vi)$\Ra$(vii)$\Ra$(viii) are trivial.

(viii)$\Ra$(i) Assume that $T:Y\to X$ is an $R$-quotient mapping from a $k_\IR$-space $Y$. To show that also $X$ is a $k_\IR$-space, let $f:X\to \IR$ be a $k$-continuous map. Then the composition $f\circ T$ is also $k$-continuous. Since $Y$ is a $k_\IR$-space, we obtain that $f\circ T$ is continuous. Since $T$ is $R$-quotient it follows that $f$ is continuous. Thus $Y$ is a $k_\IR$-space.
\end{proof}


\begin{corollary} \label{c:open-sur-kR}
A quotient group of a topological group $G$ which is a $k_\IR$-space is also a $k_\IR$-space.
\end{corollary}
\begin{proof}
The assertion follows from Theorem \ref{t:characterization-kR} and the fact that quotient maps of topological groups are open, see Theorem 5.26 of \cite{HR1}.
\end{proof}

Let $X$ be a space and $L\subseteq X$. We shall say that $L$ is {\em $k_\IR$-dense} in $X$ if $L$ is $\KK_\IR(X)$-dense in $X$, where $\KK_\IR(X)$ is the family of all $k_\IR$-subspaces of $X$. Analogously, $L$ is {\em $k$-dense} in $X$ if $L$ is $\KK(X)$-dense in $X$. A subset $L$ of a topological space $X$ is $k$-dense if, and only if, for every $x\in X$, there is $M\subseteq L$ such that $x\in \overline{M}$ and $\overline{M}$ is compact. Any $k$-dense subset is $k_\IR$-dense, and each $k_\IR$-dense subset is dense.

Now we characterize $k_\IR$-spaces  using covers of the space by $k_\IR$-subspaces (being motivated by the condition $(\star)$ from Theorem 2.5 of \cite{Gabr-B1}). The equivalence (a)$\LRa$(iv) in the next theorem was obtained by Banakh in Theorem 3.3.5 of \cite{Banakh-Survey}.

\begin{theorem}\label{t:characterization-kRe}
For a space $X$, the following assertions are equivalent:
\begin{enumerate}
\item[{\rm(a)}] $X$ is a $k_\IR$-space;
\item[{\rm(b)}] there is a cover $\mathcal{M}$ of $X$ containing $k_\IR$-spaces which has one of the following properties:
\begin{enumerate}
\item[{\rm(i)}] for every subset $A$ of $X$ and each point $x\in \overline{A}$, there is $B\in \mathcal{M}$ such that $x\in B\cap \overline{A\cap B}$;
\item[{\rm(ii)}] for every non-closed subset $Q$ of $X$, there are a point $x\in \overline{Q}\SM Q$ and an $M\in \mathcal{M}$ such that $x\in M\cap \overline{Q\cap M}$;
\item[{\rm(iii)}] $M$ is $k$-dense in $X$ for each $M\in\mathcal{M}$;
\item[{\rm(iv)}] $M\cap M'$ is dense in $X$ for all $M,M'\in\mathcal{M}$;
\item[{\rm(v)}] $M$ is $k_\IR$-dense in $X$ for each $M\in\mathcal{M}$;
\item[{\rm(vi)}]  $\mathcal{M}$ is locally finite and all $M\in\mathcal{M}$ are closed.
\end{enumerate}
\end{enumerate}
\end{theorem}

\begin{proof}
(a)$\Ra$(b) is clear (just take $\mathcal{M}=\{X\}$).  Now we prove (b)$\Ra$(a).
\smallskip

(i)$\Ra$(ii) is trivial.

(ii)$\Ra$(a) Propositions \ref{p:generated-family} and \ref{p:generated-sfg} applied to the family $\mathcal{M}$ imply that $X$ is strongly functionally generated by the family $\mathcal{M}$. Thus $X$ is a $k_\IR$-space by (i)$\LRa$(iv) of Theorem \ref{t:characterization-kR}.

(iii)$\Ra$(v) is clear because any $k$-dense subset is $k_\IR$-dense.

(iv)$\Ra$(v) To show that each $M\in \mathcal{M}$ is $k_\IR$-dense in $X$, let $x\in X$. Choose a $k_\IR$-subspace $M'\in \mathcal{M}$ such that $x\in M'$. By assumption on $\mathcal{M}$, the intersection $M\cap M'$ is dense in $X$. Therefore $x\in M'\cap \overline{M'\cap M}$. Thus $M$ is $k_\IR$-dense in $X$, as desired.

(v)$\Ra$(a) Proposition \ref{p:M-sfg} implies that the space $X$ is strongly functionally generated by the family $\mathcal{M}$ of $k_\IR$-subspaces of $X$. Thus $X$ is a $k_\IR$-space by  (i)$\LRa$(iv) of Theorem \ref{t:characterization-kR}.

(vi)$\Ra$(a)  Propositions \ref{p:generated-loc_fin} implies that $X$ is generated by the family $\mathcal{M}$. Therefore, by Proposition \ref{p:generated-sfg},  $X$ is strongly functionally generated by the family $\mathcal{M}$. Now Theorem \ref{t:characterization-kR} applies.
\end{proof}

\begin{theorem}\label{t:home-kRe}
Let $H$ be a $k_\IR$-dense (for example, $k$-dense) subspace of a homogenous space $G$. If $H$ is a $k_\IR$-space, then $G$ is also a $k_\IR$-space.
\end{theorem}
\begin{proof}
Proposition \ref{p:homeo-sfg} implies that $X$ is strongly functionally generated by the family of $k_\IR$-subspaces of $X$. It follows from Theorem \ref{t:characterization-kR} that $X$ is a $k_\IR$-space.
\end{proof}





\section{Characterizations of $s_\IR$-spaces} \label{sec:s-R-spaces}


In this section we characterize $s_\IR$-spaces analogously to characterizations of $k_\IR$-spaces obtained in Section \ref{sec:k-R-spaces}. Our first characterization uses $R$-quotient mappings and (strongly) functionally generated families. 

\begin{theorem} \label{t:characterization-sR}
For a Tychonoff space $X$, the following assertions are equivalent:
\begin{enumerate}
\item[{\rm(i)}] $X$ is an $s_\IR$-space;
\item[{\rm(ii)}]  $X$ is strongly functionally generated by the family $\mathcal{S}(X)$;
\item[{\rm(iii)}]  $X$ is functionally generated by the family $\mathcal{S}(X)$;
\item[{\rm(iv)}] $X$ is strongly functionally generated by a family of $s_\IR$-subspaces of $X$;
\item[{\rm(v)}]  $X$ is an $R$-quotient image of $\mathbf{s}\times D$, where $D$ is a discrete space;
\item[{\rm(vi)}]  $X$ is an $R$-quotient image  of a metrizable locally compact space;
\item[{\rm(vii)}]  $X$ is an $R$-quotient  image of a sequential space;
\item[{\rm(viii)}]  $X$ is an $R$-quotient  image of an $s_\IR$-space.
\end{enumerate}
\end{theorem}

\begin{proof}
The equivalence (i)$\Leftrightarrow$(ii) is essentially the definition of $s_\IR$-spaces.

(ii)$\Leftrightarrow$(iii) follows from Proposition \ref{p:fg=sfg-C-embedded} because each convergent sequence with its limit point in $X$ (as a compact subset)  is $C$-embedded in $X$.

(ii)$\Ra$(iv) is trivial.

(iv)$\Ra$(viii) Let $\mathcal{M}$ be a family of $s_\IR$-subspaces such that $X$ is strongly functionally generated by $\mathcal{M}$. Then, by Proposition \ref{p:R-quotient-sfg}, $X$ is an $R$-quotient image of the discrete sum $\bigoplus\mathcal{M}$. It remains to note that $\bigoplus\mathcal{M}$ is also an  $s_\IR$-space.
\smallskip

(ii)$\Ra$(v) follows from Proposition \ref{p:R-quotient-sfg} (applied to $\mathcal{M}=\mathcal{S}(X)$ and $D=|\mathcal{S}(X)|$).
\smallskip

The implications (v)$\Ra$(vi)$\Ra$(vii)$\Ra$(viii) are trivial.
\smallskip

(viii)$\Ra$(i) Let $f:Y\to \IR$ be an $s$-continuous map. Then the composition $f\circ T$ is also $s$-continuous. Since $X$ is an $s_\IR$-space, we obtain that $f\circ T$ is continuous. As $T$ is $R$-quotient it follows that $f$ is continuous. Thus $Y$ is an $s_\IR$-space.
\end{proof}


\begin{corollary} \label{c:open-sur-sR}
A quotient group of a topological group $G$ which is an $s_\IR$-space is an $s_\IR$-space, too.
\end{corollary}
\begin{proof}
The assertion follows from Theorem \ref{t:characterization-sR} and the fact that quotient maps of topological groups are open, see Theorem 5.26 of \cite{HR1}.
\end{proof}

Let $X$ be a space and $L\subseteq X$. We shall say that $L$ is {\em $s_\IR$-dense} in $X$ if $L$ is $\cM$-dense in $X$, where $\cM$ is the family $\mathcal{S}_\IR(X)$ of all $s_\IR$-subspaces of $X$. Any sequentially dense subset is  $s_\IR$-dense, and each $s_\IR$-dense subset is dense.

Now we characterize $s_\IR$-spaces  using covers of the space by $s_\IR$-subspaces, it is similar to Theorem \ref{t:characterization-kRe}.

\begin{theorem}\label{t:characterization-sRe}
For a space $X$, the following assertions are equivalent:
\begin{enumerate}
\item[{\rm(a)}] $X$ is an $s_\IR$-space;
\item[{\rm(b)}] there is a cover $\mathcal{M}$ of $X$ containing $s_\IR$-spaces which has one of the following properties:
\begin{enumerate}
\item[{\rm(i)}] for every subset $A$ of $X$ and each point $x\in \overline{A}$, there is $B\in \mathcal{M}$ such that $x\in B\cap \overline{A\cap B}$;
\item[{\rm(ii)}] for every non-closed subset $Q$ of $X$, there are a point $x\in \overline{Q}\SM Q$ and an $M\in \mathcal{M}$ such that $x\in M\cap \overline{Q\cap M}$;
\item[{\rm(iii)}] $M$ is sequentially dense in $X$ for each $M\in\mathcal{M}$;
\item[{\rm(iv)}] $M$ is $s_\IR$-dense in $X$ for each $M\in\mathcal{M}$;
\item[{\rm(v)}]  $\mathcal{M}$ is locally finite and all $M\in\mathcal{M}$ are closed;
\item[{\rm(vi)}] $M\cap L$ is dense in $X$ for all $M,L\in \mathcal{M}$.
\end{enumerate}
\end{enumerate}
\end{theorem}

\begin{proof}
(a)$\Ra$(b) is clear (just take $\mathcal{M}=\{X\}$).  Now we prove (b)$\Ra$(a).
\smallskip

(i)$\Ra$(ii) is trivial.

(ii)$\Ra$(a) Propositions \ref{p:generated-family} and \ref{p:generated-sfg} applied to the family $\mathcal{M}$ imply that $X$ is strongly functionally generated by the family $\mathcal{M}$. Thus $X$ is an $s_\IR$-space by Theorem \ref{t:characterization-sR}.
\smallskip

(iii)$\Ra$(iv) is clear because any sequentially dense subset is $s_\IR$-dense.
\smallskip

(iv)$\Ra$(a) Proposition \ref{p:M-sfg} implies that the space $X$ is strongly functionally generated by the family $\mathcal{M}$ of $s_\IR$-subspaces of $X$. Therefore $X$ is an $s_\IR$-space by Theorem \ref{t:characterization-sR}.
\smallskip

(v)$\Ra$(a)  Propositions \ref{p:generated-loc_fin} implies that $X$ is generated by the family $\mathcal{M}$. Therefore, by Proposition \ref{p:generated-sfg},  $X$ is strongly functionally generated by the family $\mathcal{M}$. Now Theorem \ref{t:characterization-sR} applies.
\smallskip

(vi)$\Ra$(iv) We show that each $M\in \mathcal{M}$ is $s_\IR$-dense in $X$. Let $x\in X$. Since $\mathcal{M}$ is a cover, there is $L\in \mathcal{M}$ such that $x\in L$. Then $L$ is an $s_\IR$-space and $x\in L\cap \overline{M\cap L}=L$ (since $M\cap L$ is dense in $X$).
\end{proof}

\begin{theorem}\label{t:home-sRe}
Let $H$ be an $s_\IR$-dense (for example, sequentially dense) subspace of a homogenous space $G$. If $H$ is an $s_\IR$-space, then $G$ is also an $s_\IR$-space.
\end{theorem}

\begin{proof}
Proposition \ref{p:homeo-sfg} implies that $X$ is strongly functionally generated by the family of $s_\IR$-subspaces of $X$. By Theorem \ref{t:characterization-sRe}, $X$ is an $s_\IR$-space.
\end{proof}



Recall that a space $X$ has {\em countable $\IR$-tightness} if the family $\mathcal{C}(X)$ of all countable subspaces of $X$ functionally generates $X$.

\begin{proposition}\label{p:sR-R-tightness}
Each $s_\IR$-space $X$ has countable $\IR$-tightness.
\end{proposition}
\begin{proof}
By definition, $X$ is an $s_\IR$-space if, and only if, $X$ is functionally generated by the family  $\mathcal{S}(X)$ of all convergent sequences of $X$ with their limits. Therefore, the family of all countable subspaces of $X$ functionally generates $X$.
\end{proof}

In Theorem of \cite{Blasco-78}, Blasco proved that an open subspace of a $k_\IR$-space is a $k_\IR$-space, too. Below we show that an analogous result is valid also for $s_\IR$-spaces.
\begin{proposition} \label{p:open-sR}
An open subspace $U$ of an $s_\IR$-space $X$ is an $s_\IR$-space.
\end{proposition}
\begin{proof}
By Theorem \ref{t:characterization-sR}, there exists a metrizable locally compact space $Y$ and a continuous $R$-quotient surjective map $f:Y\to X$. Put $V=f^{-1}(U)$. It follows from \cite[Proposition 1.6]{Okunev1990} that the map $g=f\rst V: V\to U$ is an $R$-quotient map. Therefore, $U$ is an $R$-quotient image of a metrizable locally compact space. Thus, by Theorem \ref{t:characterization-sR}, $U$ is an $s_\IR$-space.
\end{proof}

\begin{example} \label{exa:w-boun-sR}
Let $C$ be a metrizable compact space, and let $\xxx_\ast\in C^{\w_1}$.
Then $X=\big(C^{\w_1}\SM \{\xxx_\ast\}\big)^{\w_1}$ is an $s_\IR$-space.
\end{example}
\begin{proof}
By Proposition \ref{p:noble:fc-prod}, the space $C^{\w_1}$ is an $s_\IR$-space. It follows from Proposition \ref{p:open-sR} that $Y=C^{\w_1}\setminus \{\xxx_\ast\}$ is an $s_\IR$--space. The space $Y$ is pseudocompact (see Theorem 1.7.7 from \cite{Pseudocompact-1-2018}).
Thus, by Proposition \ref{p:noble:pseudo-prod},  $X=Y^{\w_1}$ is an $s_\IR$-space.
\end{proof}

\begin{remark} \label{rem:sR-perman}
(i) It follows from (i) of Remark \ref{rem:kR-perman} that in general the product of two $s_\IR$-spaces may be not an $s_\IR$-space.

(ii) The space $\IR^{\omega_1}$ is  an $s_\IR$-space (Proposition \ref{p:noble:fc-prod}). Since $\omega_1+1$ is closely embedded into $\IR^{\omega_1}$, it follows that even compact subspaces of $s_\IR$-spaces can be not $s_\IR$-spaces.
\end{remark}






\section{Strongly sequentially separable spaces} \label{sec:sss}


Recall that a space $X$ is called {\em strongly sequentially separable} if $X$ is separable and each countable dense subspace of $X$ is sequentially dense. Strongly sequentially separable are thoroughly studied in \cite{GLoM}.

\begin{proposition}\label{p:sR-sss}
Let $X$ be a separable space whose all countable subsets are Fr\'{e}chet--Urysohn. Then $X$ is a strongly sequentially separable $s_\IR$-space.
\end{proposition}
\begin{proof}
Let $D$ be a dense countable subspace of $X$. For every $x\in X$, set $D_x:=D\cup \{x\}$. Since $D_x$ is countable, the space $D_x$ is Fr\'{e}chet--Urysohn.
Therefore there is a sequence in $D_x$ converging to $x$. Thus $D$ is sequentially dense in $X$. Therefore $X$ is a strongly sequentially separable space.

The family $\cM=\{ D_x:x\in X\}$ consists of Fr\'{e}chet--Urysohn subsets. Since $M\cap L\supseteq D$ is dense in $X$ for any $M,L\in\cM$ and each $M\in\cM$ is an $s_\IR$-space, Theorem \ref{t:characterization-sRe}(vi) implies that $X$ is an $s_\IR$-space.
\end{proof}

Let $X$ be a space. Recall that the {\em local character $\chi(x,X)$} of $X$  at a point $x$ is the least cardinality of a local base at $x$. The {\em character} $\chi(X)$
 is the least upper bound of the local characters. Recall also that the small cardinal $\mathfrak{p}$ is defined as follows
\[
\begin{aligned}
\mathfrak{p}=\big\{ \mathcal{B}: & \;\mathcal{B} \mbox{ is a subbase for a free filter on $\w$, and there is no infinite $A$}\\
& \; \mbox{such that $A\subseteq^\ast B$ for all $B\in \mathcal{B}$}\big\},
\end{aligned}
\]
where $A\subseteq^\ast B$ means that $A\SM B$ is finite and $B\SM A$ is infinite.

\begin{theorem} \label{t:separable-p-ks}
Let $X$ be a separable space such that $\chi(X)<\mathfrak{p}$. Then $X$ is a strongly sequentially separable  $s_\IR$-space.
\end{theorem}
\begin{proof}
By Lemma 1.11 of  \cite{Nyikos1992}, each countable space whose character does not exceed $\mathfrak{p}$ is a Fr\'{e}chet--Urysohn space. It follows that every countable subspace of $X$ is a Fr\'{e}chet--Urysohn space. It remains to apply Proposition \ref{p:sR-sss}.
\end{proof}

Now we consider $C_p$-spaces which are strongly sequentially separable.

\begin{proposition}\label{p:cp-sep}
For a space $X$, the following conditions are equivalent:
\begin{enumerate}
\item[{\rm(i)}] $C_p(X)$ is separable;
\item[{\rm(ii)}] $X$ admits a condensation onto a separable metrizable space;
\item[{\rm(iii)}] $X$ is a submetrizable space and $|X|\leq\mathfrak{c}$.
\end{enumerate}
\end{proposition}
\begin{proof}
The equivalence (i)$\LRa$(ii) is well-known, see Theorem I.1.4 of \cite{Arhangel}. The implication (ii)$\Ra$(iii) is clear.

(iii)$\Ra$(ii) Let $f: X\to Y$ be a condensation onto a metrizable space. Then $w(Y)\leq|Y|\leq \mathfrak{c}$. Let $Y_*$ be a discrete space of cardinality $\mathfrak{c}$ and $Z=Y\oplus Y_*$. Then $Z$ is a metrizable space and $w(Z)=\mathfrak{c}$. By Theorem 3.4 of \cite{OsipovPytkeev2023}, the space $Z$ admits a condensation onto the Hilbert cube $[0,1]^\om$. Consequently, $Y$ and $X$ admit a condensation onto a separable metrizable space.
\end{proof}

\begin{proposition} \label{p:sss-characterization}
If $C_p(X)$ is separable, then the following conditions are equivalent:
\begin{enumerate}
\item[{\rm(i)}] $C_p(X)$ is strongly sequentially separable;
\item[{\rm(ii)}] each countable subset of $C_p(X)$ is Fr\'{e}chet--Urysohn.
\end{enumerate}
\end{proposition}

\begin{proof}
(i)$\Ra$(ii) Let $S$ be a countable subset of $C_p(X)$. Since $C_p(X)$ is separable, $S$ lies in a countable dense subset of $D\subseteq C_p(X)$. Put $\varphi=\bigtriangleup D: X\to \IR^D$, $\varphi(x):=\big(f(x)\big)_{f\in D}$, and $Y:=\varphi(X)$. Since $D$ is dense in $C_p(X)$, the map $\varphi: X\to Y$ is a condensation of $X$ onto a separable metrizable space $Y$. By Theorem 16 of \cite{GLoM}, $Y$ has the property $(\gamma)$.  Therefore $C_p(Y)$ is Fr\'{e}chet--Urysohn. Since $D\subseteq \varphi^{\#}(C_p(Y))$ and $\varphi^{\#}\big(C_p(Y)\big)$ is homeomorphic to $C_p(Y)$, we obtain that $D$ is Fr\'{e}chet--Urysohn.

(ii)$\Ra$(i) follows from Proposition \ref{p:sR-sss}.
\end{proof}

\begin{theorem} \label{t:cp-sms}
Let $X$ be a submetrizable space such that $|X|<\mathfrak{p}$. Then $C_p(X)$ is strongly sequentially separable.
\end{theorem}

\begin{proof}
Proposition \ref{p:cp-sep} implies that $C_p(X)$ is separable. Since $|X|<\mathfrak{p}$, Theorem I.1.1 of  \cite{Arhangel} implies that $\chi(C_p(X))<\mathfrak{p}$. It follows from Theorem \ref{t:separable-p-ks} that $C_p(X)$ is  strongly sequentially separable.
\end{proof}

\begin{proposition} \label{p:sss-fu}
Let  $C_p(X)$ be a strongly sequentially separable space. 
Then the following conditions are equivalent:
\begin{enumerate}
\item[{\rm(i)}] $X^n$ is Lindel\"{o}f for each $n\in\w$;
\item[{\rm(ii)}] $C_p(X)$ has countable tightness;
\item[{\rm(iii)}] $C_p(X)$ is Fr\'{e}chet--Urysohn.
\end{enumerate}
\end{proposition}
\begin{proof}
The equivalence (i)$\LRa$(ii) is the well-known Arhangelskii--Pytkeev theorem \cite[II.1.1]{Arhangel}. The implication (iii)$\Ra$(ii) is clear.

(ii)$\Ra$(iii) Observe that a space is  Fr\'{e}chet--Urysohn if, and only if, it has countable tightness and each of its countable subspaces is  Fr\'{e}chet--Urysohn. Now, since $C_p(X)$ is strongly sequentially separable, Propositions \ref{p:sss-characterization} implies that each countable subspace of $C_p(X)$ is Fr\'{e}chet--Urysohn. By (ii),  $C_p(X)$ has countable tightness. Thus $C_p(X)$ is Fr\'{e}chet--Urysohn.
\end{proof}



\section{Spaces in which $s$-continuous functions are $k$-continuous} \label{sec:sk}


Studying $s_\IR$-spaces and $k_\IR$-spaces it is natural to consider the case when $s$-continuous functions are also $k$-continuous. This motivate us to introduce the following class of spaces.

\begin{definition} {\em
A space $X$ is called an {\em $sk$-space} if every $s$-continuous function on $X$ is $k$-continuous.}
\end{definition}

In the spirit of cover-type characterizations of $s_\IR$-spaces and $k_\IR$-spaces we characterize $sk$-spaces as follows.

\begin{proposition} \label{p:sk-space-charact}
A space $X$ is an $sk$-space if, and only if,  every compact subset $K$ of $X$ is contained in a subspace $M_K$ of $X$ that is an $sk$-space.
\end{proposition}

\begin{proof}
The necessity is clear (set $M_K:=X$). To prove the sufficiency, let $f: X\to\IR$ be an $s$-continuous function. To show that $f$ is $k$-continuous, let $K\subseteq X$ be compact. Choose $M_K\subseteq X$ such that $K\subseteq M_K$ and $M_K$ is an $sk$-space. Since $f{\restriction}_{M_K}$ is $s$-continuous, $f{\restriction}_{M_K}$ is also $k$-continuous. Therefore, $f{\restriction}_{K}$ is continuous.
\end{proof}

Below we summarize some elementary properties of $sk$-spaces.
\begin{proposition} \label{p:sk-space}
\begin{enumerate}
\item[{\rm(i)}] Each $s_\IR$-space is trivially an $sk$-space.
\item[{\rm(ii)}]  An $sk$-space is an  $s_\IR$-space if, and only if, it is a  $k_\IR$-space.
\item[{\rm(iii)}] There are $k_\IR$-space spaces which are not  $sk$-spaces.
\item[{\rm(iv)}]  If $X$ has no infinite compact subsets, then $X$ is an $sk$-space. Consequently, there are $sk$-spaces which are not sequentially Ascoli.
\item[{\rm(v)}]  If every compact subset of $X$ is Fr\'{e}chet--Urysohn (for example, $X$ is angelic), then $X$ is even a hereditary $sk$-space.
\end{enumerate}
\end{proposition}

\begin{proof}
(i) and (ii) follow from the corresponding definitions.

(iii) The compact space $\beta\NN$ is a  $k_\IR$-space which is not an $sk$-space because any function on $\beta\NN$ is $s$-continuous.

(iv) Let $x\in\beta\NN \SM\NN$. Then the space $X=\{x\}\cup\NN$ is an $sk$-space. By Corollary 2.2 of \cite{Gabr-lc-Ck}, the space $Y$ is not sequentially Ascoli. In particular, $X$ is not a $k_\IR$-space.

(v) follows from Proposition \ref{p:sk-space-charact} in which $M_K=K$ for every $K\in\KK(X)$.
\end{proof}

\begin{proposition} \label{p:sk:cp}
Let $Y$ be a dense subspace of a space $X$, and assume that $Y$ satisfies one of the following conditions:
\begin{enumerate}
\item[{\rm(i)}] $Y$ is countable (that is, $X$ is separable);
\item[{\rm(ii)}] $Y$ is a Lindel\"{o}f $\Sigma$-space;
\item[{\rm(iii)}] $\mathrm(MA+\lnot CH)$ $Y^n$ is Lindel\"{o}f for each $n\in\w$;
\item[{\rm(iv)}] $\mathrm(PFA)$ $Y$ is Lindel\"{o}f.
\end{enumerate}
Then $C_p(X)$ is a hereditary $sk$-space.
\end{proposition}
\begin{proof}
By Proposition \ref{p:sk-space}(ii), it suffices to prove that each compact subspace $K$ of $C_p(X)$ is Fr\'{e}chet--Urysohn. Since $C_p(X)$ is condensed onto some subset of $C_p(Y)$, $K$ is homeomorphic to some compact subset of $C_p(Y)$.

(i) Since $C_p(Y)$ is a metrizable space,  $K$ is even metrizable.

(ii) By Proposition IV.9.10 of  \cite{Arhangel}, $K$ is a Gul'ko compact space. Thus $K$ is Fr\'{e}chet--Urysohn.

(iii) By the Arhangelskii--Pytkeev theorem \cite[II.1.1]{Arhangel}, the space $C_p(Y)$  has countable tightness. Therefore also $K$ has countable tightness. It follows from Theorem 5 of \cite{Okunev1995} that every separable subset of $K$ is metrizable. Thus $K$ is Fr\'{e}chet--Urysohn.

(iv) By Theorem IV.11.14 of \cite{Arhangel}, the compact space $K$ has countable tightness. By Theorem 1.8 of \cite{OkunevReznichenko2007}, it follows that every separable subset of $K$ is metrizable. Thus $K$ is Fr\'{e}chet--Urysohn.
\end{proof}





\section{Probability measures on $C_p(X,[-1,1])$, universally $\kappa$-measurable and universal $\kappa$-measure zero spaces} \label{sec:func-cpb}


In this section we prove three propositions which allows us to study $C_p$-spaces in Section \ref{sec:Cp-kR-sR}.

\begin{proposition} \label{p:mu-P-sigma-discrete}
Let $X$ be a space, and let $\mu\in P(\beta X)$. Then $\mu$ is continuous on $C_p(X,[-1,1])$ if, and only if, $\mu\in P_d(X)$.
\end{proposition}
\begin{proof}
Let $I=[-1,1]$ and $Y=C_p(X,[-1,1])$.

$(\La)$
Since $\mu\in P_d(X)$ then $\mu=\sum_{n=1}^\infty \lambda_n \delta_{x_n}$, where $x_n\in X$, $\delta_{x_n}$ is the Dirac measure at $x_n$, $\lambda_n\geq 0$ and $\sum_{n=1}^\infty \lambda_n=1$. Each measure $\delta_x$ is continuous and bounded by unity on $Y$. Then the series $\sum_{n=1}^\infty \lambda_n \delta_{x_n}$ converges uniformly to $\mu$. Therefore, $\mu$ is continuous.

$(\Ra)$ Since $\mu$ is continuous on $Y$ and $Y$ is dense in $I^X$, it follows from Factorization Lemma 0.2.3 of \cite{Arhangel} that there exists a countable $A\subseteq X$, a continuous function $\psi\: C_p(X|Y,I)\to \IR$, such that $\mu\rst Y=\psi\circ \pi_A$, where
\[
\pi_A\: Y \to C_p(X|Y,I),\ f \mapsto f{\restriction}_{A}
\]
is the projection and $C_p(X|A,I)=\pi_A(Y)$. If $\mu(A)=1$, then $\mu\in P_d(X)$.

We prove that $\mu(A)=1$. Assume the contrary, that is, $\mu(A)<1$. Then there exist $\tilde\mu,\nu\in P(\beta X)$ such that $\mu = (1-\lambda)\tilde\mu + \lambda\nu$, $\tilde\mu(A)=1$ and $\nu(A)=0$, where $\lambda=1-\mu(A)>0$. Since $A\subseteq X$ is countable, we have $\tilde\mu\in P_d(X)$. It follows from the proved implication $(\La)$ that $\tilde\mu\rst Y$ is continuous. Consequently, the functions $\nu{\restriction}_Y$ and
\[
\varphi\: C_p(X|A,I)\to \IR,\ f\mapsto \tfrac{1}{\lambda} \big(\psi(f) - (1-\lambda)\tilde\mu(f)\big)
\]
are continuous. Then $\nu{\restriction}_Y= \varphi\circ \pi_A$. Let $\mathbf{0}_A$ be the zero function on $A$. Since $\mathbf{0}_A\in C_p(X|A,I)$ and $\varphi(\mathbf{0}_A)=0$, there exists a finite $B\subseteq A$ and $\e>0$ such that $\varphi(W)\subseteq (-\tfrac{1}{2},\tfrac{1}{2})$ for the neighborhood
\[
W = \big\{ f\in C_p(X|A,I) : f(B)\subseteq (-\e,\e) \big\}
\]
of $\mathbf{0}_A$. Since $\nu(B)=0$, there exists a neighborhood $U$ of $B$ in $\bt X$ such that $\nu(U)<\tfrac{1}{2}$. Let $\hat g\in C_p(\beta X, [0,1])$ be a function such that $\hat g(B)=\{0\}$ and $\hat b(\beta X\setminus U)=1$. Put $g=\hat g{\restriction}_X$. Then $g(B)=\{0\}$ and $\nu(g)\geq \nu(\beta X\setminus U)> \tfrac{1}{2}$. On the other hand, since $g(B)=\{0\}$, we have $\pi_A(g)\in W$ and therefore $\nu(g)=\varphi\big(\pi_A(g)\big)\in (-\tfrac{1}{2},\tfrac{1}{2})$. This contradiction finishes the proof.
\end{proof}

\begin{proposition} \label{p:continuity-measure-compact-1}
Let $X$ be  a space, and let $\mu\in P(X)$. Then $\mu$ is $k$-continuous on $C_p\big(X,[-1,1]\big)$.
\end{proposition}
\begin{proof}
Let $K\subseteq C_p\big(K,[-1,1]\big)$ be a compact subspace. Since $\mu\in P(X)$, there is a $\sigma$-compact subset $Q$ of $X$ such that $\mu(Q)=1$. Let $Q=\bigcup_{n\in\w} Q_n$, where $Q_n$ is compact and $Q_n\subseteq Q_{n+1}$ for each $n\in\w$. Set $Q_{-1}:=\emptyset$, and for every $n\in\w$, define $\lambda_n :=\mu(Q_n\SM Q_{n-1})$ and
\[
\pi_n: C_p(X) \to C_p(Q_n), \quad \pi_n(f):=f\rst{Q_n}.
\]
For every $n\in\w$, choose $\mu_n\in P(Q_n\SM Q_{n-1})$ such that $\mu\rst{Q_n\SM Q_{n-1}} =\lambda_n \mu_n$. Recall that if $Z$ is a compact space, then, by the Eberlein--\v{S}mulyan--Grothendieck theorem \cite[3.139]{fabian-10}, the topology of pointwise convergence coincides with the weak topology on all uniformly bounded weakly pointwise compact subsets of the Banach space $C(Z)$. Therefore, for every $n\in\w$ and  considering $\mu_n$ as a continuous functional of the Banach space $C(Q_n)$, we obtain that $\mu_n$ is continuous on the compact set $\pi_n (K)\subseteq C_p(Q_n,[-1,1])$. Since $\mu_n(h)=\mu_n(\pi_n(h))$ for $h\in C_p(X,[-1,1])$, we obtain that $\mu_n$ is continuous on $K$ and $\|\mu_n{\restriction}_{K}\|\leq 1$. As $\sum_{n\in\w} \lambda_n =1$, it easily follows that the uniformly convergent series
\[
\mu\rst{K} =\sum_{n\in\w} \lambda_n \mu_n{\restriction}_{K}
\]
is continuous.
\end{proof}

\begin{proposition} \label{p:continuity-measure-compact}
Let $X$ be  a space, and let $\mu\in P(\beta X)$. Then $\mu$ is $s$-continuous on $C_p\big(X,[-1,1]\big)$ if, and only if, $\mu\in P_\sigma(X)$.
\end{proposition}

\begin{proof}
$(\Ra)$ Suppose for a contradiction that $\mu\notin P_\sigma(X)$. Then $\mu(K)>0$ for some zero-set $K\subset \beta X\SM X$. Therefor, we can find a sequence $(f_n)_n\subseteq C_p(\beta X,[0,1])$ that converges pointwise to the characteristic function $\chi_K$ of $K$. Put $g_n:=f_n\rst{X}$. Then $(g_n)_n$ converges pointwise to the zero function $\mathbf{0}_X$ on $X$ and $\lim_{n\to\infty} \mu(g_n)=\mu(K)>0$. Therefore $\mu$ is discontinuous on the sequence $\{\mathbf{0}_X\}\cup\{g_n:n\in\w\}$. A contradiction.
\smallskip

$(\La)$ Let $K\subseteq C_p\big(K,[-1,1]\big)$ be a metrizable compact subspace. Set $Y:=C_p\big(K,[-1,1]\big)$ and define
\[
j: X\to Y, \quad j(x)(f):=f(x).
\]

The Banach space $C(K)$ is Polish. Therefore the space $C_p(K)$ is Souslin, and hence its closed subspace $Y$ is Souslin as well. Therefore, by Proposition \ref{p:px:3}, $Y$ is universally measurable.
From Proposition \ref{p:px:ci} it follows that $\nu=P(\beta j)(\mu)\in P(Y)$. The measure $\nu$ is determined by the fact that
\[
\nu(g) = \mu(g\circ j)\ \ \text{ for any }g\in C_p(Y,[-1,1]).
\]
Let
\[
i:K\to C_p(Y,[-1,1]),\ \ i(f)(h):=h(f),
\]
be the canonical embedding. Then $i(K)$ is a compact subset of $C_p(Y,[-1,1])$.
From Proposition \ref{p:continuity-measure-compact-1} it follows that the restriction $\nu\rst{i(K)}$ of $\nu$ onto $i(K)$ is continuous.

We claim that $\mu{\restriction}_K = \nu{\restriction}_{i(K)} \circ i$, that is, $(\nu\circ i)(f)=\mu(f)$ for $f\in K$. For every $x\in X$, we have
\[
(i(f)\circ j)(x)=i(f)(j(x))=j(x)(f)=f(x).
\]
Therefore $i(f)\circ j=f$. Hence
\[
(\nu\circ i)(f)= \nu(i(f))=\mu(i(f)\circ j)=\mu(f)\quad (f\in K)
\]
that proves the claim.

Since $\mu{\restriction}_K =  \nu{\restriction}_{i(K)} \circ i$ and $\nu{\restriction}_{i(K)}$ is continuous, it follows that $\mu$ is continuous on $K$.
\end{proof}

For a space $X$, we set
\[
P_\kappa(X):=\big\{ \mu\in P(\beta X): \mu\text{ is $k$-continuous on }C_p\big(X,[-1,1]\big)\big\}.
\]
Propositions \ref{p:continuity-measure-compact-1} and \ref{p:continuity-measure-compact} immediately imply the following assertion.

\begin{proposition} \label{p:psi=pka}
Let $X$ be a space. Then:
\begin{enumerate}
\item[{\rm(i)}] $P(X)\subseteq P_\kappa(X) \subseteq P_\sigma(X)$;
\item[{\rm(ii)}]  if $C_p\big(X,[-1,1]\big)$ is an $sk$-space, then $P_\kappa(X)=P_\sigma(X)$.
\end{enumerate}
\end{proposition}


Being motivated by the notions of universally $\si$-measurable spaces and universal $\si$-measure zero spaces, we shall say that a space $X$ is
\begin{enumerate}
\item[$\bullet$] {\em universally $\kappa$-measurable} if $P_\kappa(X)=P(X)$;
\item[$\bullet$]  {\em universal $\kappa$-measure zero} if $P_\kappa(X)=P_d(X)$.
\end{enumerate}
Clearly, any universal $\kappa$-measure zero space is universally $\kappa$-measurable.


\section{$C_p(X)$ which are $k_\IR$-spaces and $s_\IR$-spaces} \label{sec:Cp-kR-sR}


This section is devoted to the $C_p$-part of Problem \ref{prob:Cp-kR-sR}. Following \cite{GR-kappa}, a space $X$ is called a {\em $\kappa$-space} if it is closely embedded into $C_p(Y)$ for some $k_\IR$-space $Y$. For numerous characterizations of $\kappa$-spaces, see \cite{GR-kappa}.

\begin{proposition}\label{p:Cp-kR-necessary}
If $C_p(X)$ is a $k_\IR$-space, then $X$ has the property $(\kappa)$, $X$ is universal $\kappa$-measure zero, and $X$ is a $\kappa$-space whose compact subsets are scattered.
\end{proposition}
\begin{proof}
Since $C_p(X)$ is a $k_\IR$-space and $X$ is closely embedded into $C_p\big(C_p(X)\big)$, the space $X$ is a $\kappa$-space. As $C_p(X)$ is also an Ascoli space, Theorem \ref{t:Cp-Ascoli} implies that $X$ has the property $(\kappa)$.  Let $K$ be a  compact subset of $X$. Since the property $(\kappa)$ is hereditary \cite{Sak2}, $K$ also has the property $(\kappa)$. Therefore, $K$ is scattered by Proposition 2.5 of \cite{GGKZ}.

To show that $X$ is a universal $\kappa$-measure zero space, let $\mu\in P_\kappa(X)$. By the definition of the space $P_\kappa(X)$, $\mu$ is continuous on each compact set $K\subseteq C_p\big(X,[-2,2]\big)$. Therefore, $\mu$ is continuous on each compact set $K\subseteq C_p\big(X,(-2,2)\big)$. Taking into account that $C_p(X)$ and $C_p\big(X,(-2,2)\big)$ are homeomorphic, it follows that $\mu$ is $k$-continuous on the $k_\IR$-space $C_p\big(X,(-2,2)\big)$. Therefore $\mu$ is continuous on $C_p\big(X,(-2,2)\big)$. Hence, $\mu$ is continuous on $C_p\big(X,[-1,1]\big)$. By Proposition \ref{p:mu-P-sigma-discrete}, we have $\mu\in P_d(X)$, as desired.
\end{proof}


Let $X$ and $Y$ be spaces, and let $T:X\to Y$ be a continuous mapping. Then the adjoint mapping $T^{\#}: C_p(Y)\to C_p(X)$ is continuous by Proposition 0.4.6(1) of \cite{Arhangel}. If $T$ is surjective, then, by Proposition 0.4.6(2) of \cite{Arhangel}, $T^{\#}$ is an embedding. Note also that, by Proposition 0.4.8(3) of \cite{Arhangel}, $T^{\#}$ is an embedding onto a dense subspace of $C_p(X)$ if, and only if, $T$ is a condensation.

\begin{definition}
A condensation $T:X\to Y$  is said to be a {\em $k^{\#}_\IR$-condensation} (resp., a {\em
$k^{\#}$-condensation, an $s^{\#}_\IR$-condensation or an $s^{\#}$-condensation}) if  $T^{\#}\big(C_p(Y)\big)$ is $k_\IR$-dense (resp., $k$-dense, $s_\IR$-dense or sequentially dense) in $C_p(X)$.
\end{definition}

Any $k^{\#}$-condensation is a $k^{\#}_\IR$-condensation, and each $s^{\#}$-condensation is an $s^{\#}_\IR$-condensation.

\begin{proposition} \label{p:kR-condensation-kR}
Let $T:X\to Y$ be a $k^{\#}_\IR$-condensation (in particular, a $k^{\#}$-condensation). If $C_p(Y)$ is a $k_\IR$-space, then so is $C_p(X)$.
\end{proposition}

\begin{proof}
Since $T$ is a $k^{\#}_\IR$-condensation, the $k_\IR$-space $C_p(Y)$  is homeomorphic to a $k_\IR$-dense linear subspace of $C_p(X)$. Thus, by Theorem \ref{t:home-kRe},  $C_p(X)$ is a $k_\IR$-space.
\end{proof}

\begin{proposition} \label{p:sR-condensation-sR}
Let $T:X\to Y$ be an $s^{\#}_\IR$-condensation (in particular, an $s^{\#}$-condensation). If $C_p(Y)$ is an $s_\IR$-space, then so is $C_p(X)$.
\end{proposition}

\begin{proof}
Since $T$ is an $s^{\#}_\IR$-condensation, the $s_\IR$-space $C_p(Y)$  is homeomorphic to an $s_\IR$-dense linear subspace of $C_p(X)$. Thus, by Theorem \ref{t:home-sRe},  $C_p(X)$ is an $s_\IR$-space.
\end{proof}

To obtain a concrete example of $C_p(X)$ which is a $k_\IR$-space we need the next lemma.



\begin{lemma} \label{l:one-point-k-conden}
Let $\alpha D=D\cup\{d_\infty\}$ be a one-point compactification of an infinite discrete set $D$, and let $X$ be the set $D\cup\{d_\infty\}$ endowed with a topology $\tau$ such that the identity map $T:X\to \alpha D$ is a condensation. Then  $T$ is a $k^{\#}$-condensation.
If $|X|$ is not sequential, then $T$ is an $s^{\#}_\IR$-condensation.
\end{lemma}
\begin{proof}
Let $f\in C_p(X)$. For every $a,b\in\IR$, let $I(a,b):=[a,b]$ if $a\leq b$, and $I(a,b):=[b,a]$ if $a>b$. Set
\[
K:=\big\{ h\in \IR^X: h(x)\in I\big(f(d_\infty),f(x)\big)  \; \mbox{ for all }\; x\in X\big\}.
\]
Then $K$ is a compact subset of $\IR^X$. It is easy to see that $K\subseteq C_p(X)$.
Note that $f(d_\infty)=g(d_\infty)$ for every $g\in K$.

We show that $f\in \overline{K\cap C_p(\alpha D)}$. Since $f\in K$, it suffices to verify that $K\cap C_p(\alpha D)$ is dense in $K$. Put
\[
M = \{ g\in K : |\{x\in D: g(x)\neq f(d_\infty)\}| < \omega \}.
\]
Then $M\subseteq K\cap C_p(\alpha D)$ and $M$ is dense in $K$. In particular, $f\in K\cap \overline{K\cap C_p(\alpha D)}$. Therefore $C_p(\alpha D)=T^{\#}(C_p(\alpha D))$ is $k$-dense in $C_p(X)$. Thus $T$ is a $k^{\#}$-condensation.

Assume that $|X|$ is not a sequential cardinal. The compact set $K$ is homeomorphic to the Tychonoff cube $[0,1]^\tau$, where $\tau\leq |X|$. Therefore, $\tau$ is not sequential, too.
It follows from Proposition \ref{p:noble:fc-prod} that $K$ is an $s_\IR$-space. Therefore, $C_p(\alpha D)=T^{\#}(C_p(\alpha D))$ is $s_\IR$-dense in $C_p(X)$.
\end{proof}

\begin{theorem} \label{t:exa-kR-Cp}
Let $X$ be a space with one non-isolated point. Then $C_p(X)$ is a $k_\IR$-space. If $|X|$ is not sequential, then $C_p(X)$ is an $s_\IR$-space.
\end{theorem}

\begin{proof}
Let $d_\infty$ be the non-isolated point of $X$, and let $D=X\SM \{d_\infty\}$. Let $Y=D\cup \{d_\infty\}$ be the one point compactification of the discrete space $D$. Then the identity map $T:X\to Y$ is a condensation. By Lemma \ref{l:one-point-k-conden}, $T$ is a $k^{\#}$-condensation and if $|X|$ is not a sequential cardinal, then $T$ is an $s^{\#}_\IR$-condensation. Since $Y$ is a scattered compact space, then, by Theorem III.1.2 of  \cite{Arhangel},  $C_p(Y)$ is a Fr\'{e}chet--Urysohn space. From Theorem \ref{t:home-kRe} it follows that $C_p(X)$ is a $k_\IR$-space. If $|X|$ is not sequential, then, by Theorem \ref{t:home-sRe},  the space $C_p(X)$ is an $s_\IR$-space.
\end{proof}

\begin{theorem} \label{t:lc-kR-Cp}
Let $X=\bigoplus_{\alpha\in \AAA}X_\alpha$, where all $X_\alpha$ are countable spaces. Then $C_p(X)$ is a $k_\IR$-space. If $|X|$ is not sequential, then $C_p(X)$ is an $s_\IR$-space.
\end{theorem}
\begin{proof}
Since $C_p(X_\alpha)$ is metrizable for $\alpha\in \AAA$, Proposition \ref{p:noble:fc-prod}(i) implies that $C_p(X)$ is a $k_\IR$-space.
If $|X|$ is not a sequential cardinal, then Proposition \ref{p:noble:fc-prod}(ii) implies that $C_p(X)$ is an $s_\IR$-space.
\end{proof}

\begin{theorem} \label{t:lP-kR-Cp}
Let $X=\bigoplus_{\alpha\in \AAA}X_\alpha$, where all $X_\alpha$ are Lindel\"{o}f $P$-spaces. Then $C_p(X)$ is a $k_\IR$-space. If $|\AAA|$ is not sequential, then $C_p(X)$ is an $s_\IR$-space.
\end{theorem}
\begin{proof}
Let $Y=X\cup\{x_{\infty}\}$ be a one point Lindel\"{o}fication of the space $X$: sets of the form $Y\setminus \bigcup_\alpha\{X_\alpha\in \cB\}$, where $\cB\subseteq\AAA$ and $|\cB|\leq\om$, form a base at the point $x_{\infty}$. The space $Y$ is a Lindel\"{o}f $P$-space. Therefore, by Problem 135 of  \cite{Tkachuk-2}, the space $C_p(Y)$ is Fr\'{e}chet--Urysohn. Any $\si$-product in the product $\prod_{\alpha\in \AAA}C_p(X_\alpha)=C_p(X)$ is embedded in $C_p(Y)$, so it is a Fr\'{e}chet--Urysohn space. It remains to apply Proposition \ref{p:noble:prod}.
\end{proof}

\begin{problem}
Let a space $X$ be such that $C_p(X)$ is a $k_\IR$-space and $|X|$ is not sequential. Is then $C_p(X)$ an $s_\IR$-space?
\end{problem}

\begin{problem}
Is there an example in ZFC of a   space $X$ such that $C_p(X)$ is a $k_\IR$-space and $C_p(X)$ is not an $s_\IR$-space?
\end{problem}

\begin{problem} \label{pr:Rc-sr}
Is  $\IR^{\mathfrak{c}}$ an $s_\IR$-space?
\end{problem}

Note that Problem \ref{pr:Rc-sr} is equivalent to the Keisler--Tarski problem \cite{KeislerTarski1963}, which was discussed in Section \ref{sec:s-R-spaces}.

As the statement below shows, for many cases of $C_p$-spaces, the property of being an $s_\IR$-space coincides with the property of being a $k_\IR$-space.

\begin{proposition} \label{p:cp-sR=kR}
Let a space $X$ satisfy the conditions of Proposition \ref{p:sk:cp}. Then $C_p(X)$ is a $k_\IR$-space if and only if $C_p(X)$ is an $s_\IR$-space.
\end{proposition}
\begin{proof}
It follows from Proposition \ref{p:sk:cp} that $C_p(X)$ is $sk$-space. It remains to apply Proposition \ref{p:sk-space}(ii).
\end{proof}

\begin{problem}
Is it possible to prove Proposition \ref{p:cp-sR=kR} in ZFC?
\end{problem}

Till the end of this section we consider the case when $C_p(X)$ is/or not an $s_\IR$-space.

\begin{proposition}\label{p:sR-Cp-realcompact}
\begin{enumerate}
\item[{\rm(i)}] If $X$ is an $s_\IR$-space, then $C_p(X)$ is realcompact.
\item[{\rm(ii)}] If $C_p(X)$ is an $s_\IR$-space, then $X$ is realcompact.
\end{enumerate}
\end{proposition}
\begin{proof}
(i) Since, by  Proposition \ref{p:sR-R-tightness}, the $\IR$-tightness of  $X$ is countable, Corollary~II.4.17 of \cite{Arhangel} implies that $C_p(X)$ is realcompact.

(ii) follows from (i) and the fact that $X$ is closely embedded into $C_p(C_p(X))$.
\end{proof}

\begin{theorem} \label{t:sss-sr}
Let  $C_p(X)$ be a strongly sequentially separable space. Then $C_p(X)$ is an $s_\IR$-space, and the following conditions are equivalent:
\begin{enumerate}
\item[{\rm(i)}] $X^n$ is Lindel\"{o}f for each $n\in\om$;
\item[{\rm(ii)}] $C_p(X)$ has countable tightness;
\item[{\rm(iii)}] $C_p(X)$ is Fr\'{e}chet--Urysohn.
\end{enumerate}
\end{theorem}
\begin{proof}
Propositions \ref{p:sR-sss} and \ref{p:sss-characterization} imply that $C_p(X)$ is an $s_\IR$-space. The equivalences  (i)$\LRa$(ii)$\LRa$(iii) follow from Proposition \ref{p:sss-fu}.
\end{proof}

\begin{problem} \label{pr:sms---sr=fu}
Does there exist a separable metrizable space $X$ for which $C_p(X)$ is an $s_\IR$-space which is not Fr\'{e}chet--Urysohn?
\end{problem}

The following assertion follows from Theorems \ref{t:cp-sms} and \ref{t:sss-sr}.

\begin{theorem} \label{t:cp-p-sr}
Let $X$ be a submetrizable space and $|X|<\mathfrak{p}$. Then $C_p(X)$ is a strongly sequentially separable $s_\IR$-space.
\end{theorem}

\begin{theorem} \label{t:Cp-sR-space-necessary}
If $C_p(X)$ is an $s_\IR$-space, then $X$ is a realcompact, universal $\sigma$-measure zero space and has the property $(\kappa)$.
\end{theorem}

\begin{proof}
Since $C_p(X)$ is an $s_\IR$-space, Proposition \ref{p:sR-Cp-realcompact} implies that $X$ is realcompact.
Since any $s_\IR$-space is a $k_\IR$-space,  Proposition \ref{p:Cp-kR-necessary} implies that $X$ has the property $(\kappa)$.

To show that $X$ is a universal $\si$-measure zero space, let $\mu\in P_\sigma(X)$.  Proposition \ref{p:continuity-measure-compact} implies that $\mu$ is continuous on each metrizable compact set $K\subseteq C_p\big(X,[-2,2]\big)$.
Therefore, $\mu$ is continuous on each metrizable compact set $K\subseteq C_p\big(X,(-2,2)\big)$.
Taking into account that $C_p(X)$ and $C_p\big(X,(-2,2)\big)$ are homeomorphic, it follows that $\mu$ is $s$-continuous on the $s_\IR$-space $C_p\big(X,(-2,2)\big)$. Therefore $\mu$ is continuous on $C_p\big(X,(-2,2)\big)$.
Hence, $\mu$ is continuous on $C_p\big(X,[-1,1]\big)$.
By Proposition \ref{p:mu-P-sigma-discrete}, $\mu\in P_d(X)$, as desired.
\end{proof}

Recall (see \cite[\S~40.III]{Kuratowski}) that a separable metrizable space $X$ is said to be a {\em $\lambda$-space} if every countable subset of $X$ is a $G_\delta$-set of $X$.  In ZFC, there exists a $\lambda$-space of real numbers of cardinality $\w_1$ \cite[p.~215]{Miller1984}.
A {\em Sierpi\'{n}ski set} is an uncountable subset of $\IR$ whose intersection with every measure-zero (= Lebesgue zero) set is countable. The existence of Sierpi\'{n}ski  sets is independent of the axioms of ZFC. In \cite{Sierpinski}, Sierpi\'{n}ski  showed that they exist in $\IR$ if the continuum hypothesis is true. On the other hand, they do not exist if Martin's axiom for $\aleph_1$ is true.

\begin{proposition} \label{p:Cp-Sierpinski}
If $X$ is a Sierpi\'{n}ski set, then $C_p(X)$ is an Ascoli space which is not an $s_\IR$-space. Consequently, under $(CH)$, there exists a separable metrizable space $X$ such that $C_p(X)$ is an Ascoli space but not a $k_\IR$-space.
\end{proposition}
\begin{proof}
Every Sierpi\'{n}ski set is a $\lambda$-space \cite[10.1]{Miller1984}. Hence, by Theorem 3.2 of \cite{Sak2}, $X$ has the property $(\kappa)$. Therefore, by Theorem \ref{t:Cp-Ascoli}, $C_p(X)$ is an Ascoli space.
Since a Sierpi\'{n}ski set is an uncountable subset of $\IR$ whose intersection with every Lebesgue zero set is countable, it follows that the Lebesgue measure of  $X$ is not equal to zero. Since the Lebesgue measure is continuous, the space $X$ is not a universal measure zero set. It follows from Theorem \ref{t:px:polish} that $X$ is not a universal $\sigma$-measure zero space. Therefore, by Theorem \ref{t:Cp-sR-space-necessary}, $C_p(X)$ is not an $s_\IR$-space.

It follows from Proposition \ref{p:cp-sR=kR} that $C_p(X)$ is not a $k_\IR$-space.
\end{proof}

\begin{remark}
In \cite{Mi73}, Michael constructed an $\aleph_0$-space which is a $k_\IR$-space but not a $k$-space. In Problem 6.8 of \cite{BG} the next question is posed: {\em Does there exist a cosmic (or an $\aleph_0$-space) which is Ascoli but not a  $k_\IR$-space}? Since a Sierpi\'{n}ski set is cosmic, the space $C_p(X)$ is also cosmic. Therefore, under the continuum hypothesis $\mathrm{(CH)}$, Proposition \ref{p:Cp-Sierpinski} gives a positive answer to this question.
\end{remark}

\begin{problem}
Is there an example in ZFC of a  separable metrizable space $X$ such that $X$ has the property $(\kappa)$ ($X$ is a $\lambda$-space) but not universal measure zero?
\end{problem}

\begin{problem}
Is there a universal measure zero separable metrizable space $X$ such that $C_p(X)$ is  Ascoli but not an $s_\IR$-space?
\end{problem}


\bibliographystyle{amsplain}

\end{document}